\documentclass[]{article}
\usepackage[utf8]{inputenc}
\usepackage[a4paper,top=2.5cm,bottom=2.5cm,left=2cm,right=2cm]{geometry}
\usepackage{graphicx} 
\usepackage{float}
\usepackage{textcomp}
\usepackage{xcolor}
\usepackage{amsmath}
\usepackage{amssymb}
\usepackage{amsthm}
\usepackage{amsfonts}
\usepackage[english]{babel}
\usepackage{verbatim}
\usepackage{array}
\usepackage{hyperref}
\usepackage{cases}
\usepackage{cite}
\usepackage{afterpage}
\usepackage{url}    
\usepackage{esint}    
\usepackage{caption}
\usepackage{subcaption} 
\usepackage{dsfont}                                                                                   
\newcommand{\authoraddress}[2]{%
	\textsc{#1} \textit{E-mail address:} \protect\url{#2}%
}                                                                                                                     
\newcommand{\AuthorAddressone}{                                                                                          
	\authoraddress{Institute for Applied Mathematics, University of Bonn, 53115 Bonn, Germany.}{dematte@iam.uni-bonn.de}%
} 
\newcommand{\AuthorAddresstwo}{                                                                                          
	\authoraddress{Institute for Applied Mathematics, University of Bonn, 53115 Bonn, Germany.}{velazquez@iam.uni-bonn.de}%
} 
\numberwithin{equation}{section}
\newtheorem{theorem}{Theorem}[section]
\newtheorem{lemma}{Lemma}[section]

\newtheorem{prop}{Proposition}[section]

\theoremstyle{definition}

\theoremstyle{remark}
\newtheorem*{remark}{Remark}

\newcommand{\F}{\mathcal{F}}

\newcommand{\Ss}{\mathbb{S}}
\newcommand{\R}{\mathbb{R}}

\newcommand{\eps}{\varepsilon}

\DeclareMathOperator*{\Div}{div}

\title{Well-posedness for a two-phase Stefan problem with radiation}
\author{Elena Demattè\thanks{\AuthorAddressone}, Juan J.L. Velázquez\thanks{\AuthorAddresstwo}}

\begin{document}
	
	\maketitle
	\begin{abstract}
		In this paper we consider a free boundary problem for the melting of ice where we assume that the heat is transported by conduction in both the liquid and the solid part of the material and also by radiation in the solid. Specifically, we study a one-dimensional two-phase Stefan-like problem which contains a non-local integral operator in the equation describing the temperature distribution of the solid. We will prove the local well-posedness of this free boundary problem combining the Banach fixed-point theorem and classical parabolic theory. Moreover, constructing suitable stationary sub- and supersolutions we will develop a global well-posedness theory for a large class of initial data. 
	\end{abstract}
\textbf{Acknowledgments:} The authors gratefully acknowledge the financial support of Bonn International Graduate School of Mathematics (BIGS) at the Hausdorff Center for Mathematics founded through the Deutsche Forschungsgemeinschaft (DFG, German Research Foundation) under Germany’s Excellence Strategy – EXC-2047/1 – 390685813. \\


\noindent\textbf{Statements and Declarations:} The authors have no relevant financial or non-financial interests to disclose.\\

\noindent\textbf{Data availability:} Data sharing not applicable to this article as no datasets were generated or analysed during the current study.
	\section{Introduction}
	In this paper we study a one-dimensional two-phase free boundary problem which considers the melting of ice due to conduction and radiation. Specifically, we study the situation in which $ \R^3  $ is filled by a material in its liquid and solid phase. The moving interface is the surface of contact between the liquid and the solid and it changes its position according to the melting of the solid or the solidification of the liquid. The temperature of the interface equals the melting temperature $ T_M $ of the material, while the temperature of the liquid phase is larger than $ T_M $ and the temperature of the solid phase is smaller than $ T_M $. We also assume that the heat is transferred by conduction in both phases of the material, similarly to the classical Stefan problem. In addition, we assume that in the solid phase the heat is transferred also by radiation. This is equivalent to the assumption of a transparent liquid phase, where the material does not interact with the radiation, and of an opaque solid phase, where both absorption and emission processes take place. 
	
	To be more precise, we consider a model in which at the initial time $ t=0 $ the liquid phase of the material fills the negative half-space $ \{x\in\R^3: x_1<0\} $ and the solid phase fills the positive half-space $ \{x\in\R^3: x_1>0\} $. Hence, the interface is at time $ t=0 $ the plane $ \{0\}\times \R^2 $. The model that we study is a one-dimensional free boundary problem obtained under the further assumption that the temperature depends only on the variable $ x_1 $ and that both phases have the same constant density.
	\begin{figure}[H]
	 \begin{center}
	\vspace{-0.3cm}	\includegraphics[height=5cm]{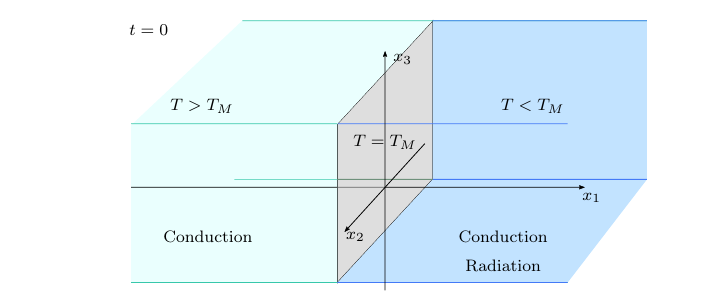}\vspace{-0.3cm}	\caption{Illustration of the considered model at the initial time $ t=0 $.}
	\end{center}	
	\end{figure}
In the case where the heat is transferred by conduction, the evolution equation for the temperature is given by the well-known heat equation
\begin{equation}\label{conduction}
	C\partial_t u=K\partial_x^2u,
\end{equation}
where $ C>0 $ is the heat capacity of the material and $ K>0 $ is the conductivity of the material.\\

When the heat is transferred also by radiation we have to include in the model, besides the terms describing heat conduction, the ones associated to the radiative transfer equation, i.e. the kinetic equation for the density of radiative energy. Let us consider first a body $ \Omega\subset \R^3 $ interacting with radiation. Defining by $ I_\nu(t,x,n) $ the radiation intensity, i.e. the density of energy carried by photons with frequency $ \nu>0 $, at position $ x\in\Omega$, moving in direction $ n\in\Ss^2 $ at time $ t>0 $, the radiative transfer equation is given by
\begin{equation*}\label{RTEfull}
	\frac{1}{c}\partial_t I_\nu(t,x,n)+n\cdot \nabla_x I_\nu(t,x,n)=\alpha_\nu^e-\alpha_\nu^aI_\nu(t,x,n)+\alpha_\nu^s\left(\int_{\Ss^2} \mathbb{K}(n,n')I_\nu(t,x,n')\;dn'-I_\nu(t,x,n)\right),
\end{equation*}
where $ \mathbb{K} $ is the scattering kernel and $ \alpha_\nu^e $, $ \alpha_\nu^a $ and $ \alpha_\nu^s $ are the emission parameter, the absorption and the scattering coefficient, respectively. They are used in order to describe the emission, absorption and scattering of photons, which are the processes involved in the interaction of radiation with the matter. In this paper we neglect the scattering process and we set $ \alpha_\nu^s=0 $. Emission and absorption are the only processes through which the radiation changes the temperature of a material, although scattering could change the spatial distribution of radiation. Moreover, we consider local thermal equilibrium, i.e. we assume that at any time $ t>0 $ and at any point $ x $ there is a well-defined temperature. Under this assumption, the emission parameter takes the specific form $ \alpha_\nu^e= \alpha_\nu^a B_\nu(T(t,x)) $, where $ B_\nu(T)= \frac{2h\nu^3}{c^2}\frac{1}{e^{\frac{h\nu}{kT}}-1} $ is the Planck distribution of a black body. Furthermore, we consider in this paper only the Grey approximation with constant absorption coefficient, i.e. we assume that the $ \alpha_\nu^a $ does not depend on the frequency $ \nu $ nor on the space variable $ x $.

Thus, defining $ \alpha_\nu^a=\alpha $, the radiative transfer equation we will study takes the form
\begin{equation}\label{RTE}
	\frac{1}{c}\partial_t I_\nu(t,x,n)+n\cdot \nabla_x I_\nu(t,x,n)=\alpha \left(B_\nu(T(t,x)) -I_\nu(t,x,n)\right).
\end{equation}
The evolution of the temperature due to the radiation process is given by the energy balance equation 
\begin{equation}\label{div-free}
	C\partial_t T(t,x)+\frac{1}{c}\partial_t\left(\int_0^\infty d\nu\int_\Ss dn\;I_\nu(t,x,n)\right)+\Div\left(\int_0^\infty d\nu\int_\Ss dn\;nI_\nu(t,x,n)\right)=0
\end{equation}
coupled with the radiative transfer equation \eqref{RTE} and with suitable boundary conditions. 

Turning back to the free boundary problem, we assume that the temperature depends only on $ x_1 $. Therefore, the interface is given by the plane $ \{s(t)\}\times \R^2 $ normal to the $ x_1 $-axis. In this paper we assume that there is no external source of radiation. Mathematically, we consider as boundary condition for the radiation at the interface
\begin{equation}\label{bnd.interface}
	I_\nu\left(t,(s(t),x_2,x_3),n\right)=0\;\;\text{ if } n_1>0,
\end{equation}
where $ n_1=n\cdot e_1 $ for $ n\in\Ss^2 $. We emphasize that radiation can escape the solid, i.e. $ I_\nu(t,s(t),n)\ne0 $ for $ n_1<0 $. Since the liquid is transparent, the escaped photons do not interact with the liquid and they do not change their direction. Specifically, the escaped radiation cannot return in the solid. This is the reason why in absence of external sources the incoming boundary condition at the interface is given by \eqref{bnd.interface}. Moreover, since we are studying a problem where the heat is transferred by both conduction and radiation we may assume that the radiation intensity solves the stationary radiative transfer equation and consequently that in \eqref{div-free} the time derivative of the total radiation energy is negligible. Indeed, since the photons travel with speed $ c $, i.e. speed of light, the radiation intensity stabilizes in a much shorter time than the characteristic time required for significant changes of the temperature due to the transport of heat by both conduction and radiation. Finally, the evolution equation for the temperature in the case of heat transfer due to conduction and radiation takes into account the heat production rate due to both processes and it is a combination of the heat equation \eqref{heat.eq} and the energy balance equation \eqref{div-free} for the radiative transfer equation, i.e. in this case
\begin{equation}\label{mixed}
	C\partial_t T(t,x)-K\partial_x^2T(t,x)+\Div\left(\int_0^\infty d\nu\int_\Ss dn\;nI_\nu(t,x,n)\right)=0.
\end{equation}
The interface moves according to the Stefan condition, i.e. the energy balance law at the interface, which is given by
\begin{equation}\label{stefan.condition}
	\dot{s}(t)=\frac{1}{L}\left(K_S \partial_{x_1} T(t,s(t)^+)-K_L\partial_{x_1} T(t,s(t)^-)\right),
\end{equation}
where $ K_S $ and $ K_L $ are the conductivity of the solid and the liquid, respectively, and L is the specific latent heat. Notice that the Stefan condition is the same as the one for the classical Stefan problem. This can be explained by the fact that the intensity of radiation $ I_\nu $ in the liquid is given by the constant continuation of the radiation intensity at the interface. Indeed, the liquid is assumed to be transparent, i.e. the radiation is still present and it passes through the liquid region without interacting with it. In other words, in the liquid the temperature evolves also according to \eqref{mixed} for $ I_\nu $ solving \eqref{RTE} with $ \alpha=0 $. Note that in the case in which $ I_\nu $ is constant, e.g. in the liquid, the divergence term disappears and \eqref{mixed} is equivalent to \eqref{conduction}. Therefore, the Stefan condition, according to which the discontinuity of heat flux at the boundary is proportional to the speed of the motion of the interface, is given by \eqref{stefan.condition} since the flux of radiating energy is continuous. Defining by $ C_S $ and $ C_L $ the heat capacity of the solid and the liquid and putting together equations \eqref{conduction}, \eqref{RTE}, \eqref{mixed}, \eqref{bnd.interface} and \eqref{stefan.condition} we study the following free boundary problem
	\begin{equation}\label{syst.1}
		\begin{cases}
			C_L \partial_t T(t,x_1)=K_L \partial_{x_1}^2 T(t,x_1)& x_1<s(t),\\
			C_S \partial_t T(t,x_1)=K_S \partial_{x_1}^2 T(t,x_1)-\Div\left(\int_0^\infty d\nu\int_{\Ss} dn nI_\nu(t,x,n)\right)& x_1>s(t),\\
			n\cdot \nabla_x I_\nu(t,x,n)=\alpha \left(B_\nu(T(t,x_1))-I_\nu(t,x,n)\right)& x_1>s(t),\\
			I_\nu(t,x,n)=0&x_1=s(t),\;n_1>0,\\
			T(t,s(t))=T_M&x_1=s(t),\\
			T(0,x)=T_0(x)& x_1\in\R,\\
			\dot{s}(t)=\frac{1}{L}\left(K_S \partial_{x_1} T(t,s(t)^+)-K_L\partial_{x_1} T(t,s(t)^-)\right).
		\end{cases}
	\end{equation}
We emphasize that the main peculiarity of the model \eqref{syst.1} is that only the solid is emitting radiation. This is due to the assumption of a perfectly transparent solid. Another interesting problem would be to consider in addition a non-trivial external source of radiation heating the solid from far away, i.e. to set as boundary conditions
\[	\left(t,(s(t),x_2,x_3),n\right)=g_\nu(n)>0\;\;\text{ if } n_1>0.\]
In this case we expect to observe superheated solid, i.e. regions in the solid phase where the temperature is greater than the melting temperature.

Finally, we mention that in the upcoming paper \cite{StefanRad2} we continue the analysis of the problem presented in this article constructing traveling wave solutions for \eqref{syst.1}, which are the natural candidates to describe the long time asymptotics for the solution to \eqref{syst.1}.
\subsection{Summary of previous results}
In this paper we consider a free boundary problem similar to the classical two-phase Stefan problem modeling the melting of ice assuming in addition that the heat is transferred by radiation in the solid. The pioneer of the study of such free boundary problems was J. Stefan, after whom these problems are named, in the late 80's (c.f. \cite{Stefan3}). The same person worked also on heat radiation developing the well-known Stefan law, otherwise called Stefan-Boltzmann law, which states that the total radiation of a black body is proportional to the fourth power of the temperature (cf. \cite{Stefan5}). In this subsection we revise important results on both the theory of free boundary problems and of radiative heat transfer. 

Starting from the work of J. Stefan, the one and two-phase Stefan problem for the melting of ice has been extensively studied in both one-dimensional and higher dimensional versions, and different definitions of solutions were considered, such as classical solutions and weak enthalpy solutions. The first are defined as strong solutions of the Stefan problem, while the latter are defined as the weak solutions of the enthalpy formulation of the free boundary problem.
 

The well-posedness theory for classical solutions to the Stefan problem has been considered for example in \cite{Friedman5,Friedman1,Rubinstein} showing the local well-posedness via a fixed-point equation for integral equations of Volterra type. The global well-posedness is proved applying the maximum principle. In \cite{Friedman1} another approach involving the Baiocchi transform is also considered. In \cite{Friedman2,FriedmanKinderlehrer} a well-posedness result based on the study of a suitable variational inequality is presented. Other results about the well-posedness theory for classical solutions can be found in \cite{CannonHenryKotlow,CannonPrimicerio,Meirmanov}. In \cite{Friedman1,Meirmanov} the asymptotic behavior of the one-dimensional one-phase problem is considered and it is proved that the temperature is given by a self-similar profile as $ t\to\infty $. 
Important results about the theory of weak (enthalpy) solutions can be found in \cite{Friedman4}, for the two-phase one-dimensional problem, and in \cite{Friedman3}, for the higher dimensional one and two-phase problem. 

An important difference between classical and weak enthalpy solutions is in the context of supercooled, superheated and mushy regions. Regions of liquid (or solid) at a temperature $ T<T_M $ (or $ T>T_M $) are denoted in the classical theory supercooled (resp. superheated) regions. In the absence of such regions, weak and classical solutions are equivalent. In the enthalpy formulation the onset of mushy regions, i.e. $ \left\{(t,x):T(t,x)=T_M\right\} $ with positive measure, is allowed, whereas supercooled or superheated regions cannot appear.

Concerning the study of mushy regions for Stefan problems with volumetric heat sources, in \cite{FasanoPrimicerio2} the authors give a clear distinction between classical and weak enthalpy solutions and introduce the notion of classical enthalpy solutions, which allow the formation of mushy regions. In \cite{LaceyShillor} the authors consider the classical solutions to a one-dimensional two-phases Stefan problem with volumetric heat sources and show the formation of regions of supercooled liquid or superheated solid. Other examples of studies of the formation of mushy regions are \cite{Bertsch,FasanoPrimicerio1,LaceyTayler,PrimicerioUghi,Visintin1,Visintin2}. 

Moving to the transfer of heat by radiation, this problem has been widely studied starting form the seminal works of Compton \cite{compton} in 1922 and of Milne \cite{Milne} in 1926. The mathematical theory behind the interaction of photons with matter deals with the study of the radiative transfer equation, as given in \eqref{RTE}. The derivation and the main properties of this kinetic equation can be found for instance in \cite{Chandrasekhar,mihalas,oxenius,Rutten,Zeldovic}. 

Also in recent years many different problems were considered, such as well-posedness results, the diffusion approximation and the combination of radiative transfer with other existing models. In \cite{dematt42024compactness,jang} the authors proved the well-posedness theory for the stationary radiative transfer equation. 
Another extensively studied problem is the so-called diffusion approximation, i.e. the limit of the radiative transfer equation when the mean free path of the photons is very small. 
See for instance \cite{Golse3,Golse6,dematte2024equilibrium,dematte2023diffusion} and the reference therein. 

The radiative heat transfer has been also considered in problems concerning more involved interaction between radiation and matter. For instance, problems studying these interactions in a moving fluid can be found in \cite{Golse1,Golse2,mihalas,Zeldovic}. 
We refer to \cite{dematte,paper} and the reference therein for problems considering models of coupled Boltzmann equations and radiative transfer. 
Moreover, also problems where the heat is transported in a body by conduction and radiation
and homogenization problems in porous and perforated domains, where the heat is transported by conduction, by radiation and in some cases also by convection, have been studied in several works. We refer to the literature of our previous work \cite{dematte2024equilibrium}.

Finally, models of melting processes assuming transport of heat by conduction and radiation has been considered numerically in some engineering applications, for instance in \cite{engineering2,engineering5,engineering1,engineering3,engineering4}. There, free boundary problems concerning phase transition due to both conduction and radiation are numerically analyzed and several one-dimensional models considering one, two, and three-phase Stefan-like problems are formulated based on experimental results. Another relevant numerical application, which has been extensively studied in recent years, is the analysis of free boundary problems modeling the vaporization of droplets where the heat is transported by radiation and conduction. For example in \cite{droplets3,droplets2,droplets5,droplets7,droplets8,droplets1,droplets6} numerical simulations show that the radiative heat transfer plays an important role in the vaporization of droplets.
\subsection{Main results and plan of the paper}
In this paper we study the well-posedness theory for problem \eqref{syst.1} and it is structured as follows.

First of all, in the next subsection \ref{Subs.scal.} we will perform some rescaling obtaining an equivalent version of the problem \eqref{syst.1} which we will consider  for the rest of the paper, while in subsection \ref{subs.notation} we clarify some notations that we will use throughout this paper.

In the following Section \ref{Sec.loc.Well.Pos}, using Banach fixed-point theorem and classical parabolic theory, we show a local well-posedness result, which can be summarized as the following
\begin{theorem}\label{thm.local.0}
	Let $ T_0\in C^{0,1}(\R) $ be the initial temperature satisfying the condition
	\[T_0(x)>T_M\;\text{ if }x<0,\quad T_0(0)=T_M,\quad T_0(x)<T_M\;\text{ if }x>0.\]
	Under suitable assumptions on the regularity of $ T_0 $ in $ \R_\pm $ there exists a time $ t_*>0 $ such that there exists a unique solution to the problem \eqref{syst.1} for $ t\in[0,t_*] $. Moreover, the temperature satisfies
	\[T_0(x)>T_M\;\text{ if }x<s(t),\quad T_0(s(t))=T_M,\quad T_0(x)<T_M\;\text{ if }x>s(t).\]
\end{theorem}
In addition to the local well-posedness theory, in Section \ref{subs.global} we will also prove a more general global well-posedness result, which applies for a large family of initial data. The following theorem will be proved constructing a suitable family of sub- and supersolutions and applying the maximum principle to the parabolic equations in \eqref{syst.1}.
\begin{theorem}[Global Well-posedness]\label{thm.global.0}
	Let $ T_0 $ as in Theorem \ref{thm.local.0}. There exists a large class of initial data for which there exists a unique global in time solution to the problem \eqref{syst.1}.
\end{theorem}
As we will see in Section \ref{subs.global}, the assumptions on the initial data concern the upper bound on the initial temperature of the liquid. 
\subsection{Some scaling}\label{Subs.scal.}
In this subsection we rescale the problem \eqref{syst.1} obtaining an equivalent problem in order to reduce the number of parameters. In this paper we will not assuming a positive source of radiation. Nevertheless, we remark that the computations we perform in this subsection can be also adapted in the presence of a non-trivial source of radiation. 

First of all we reduce the radiative term and the radiative transfer equation to a non-local integral term. To this end we solve 
\begin{equation*}
	\begin{cases}
		n\cdot \nabla_x I_\nu(t,x,n)=\alpha \left(B_\nu(T(t,x_1))-I_\nu(t,x,n)\right)& x_1>s(t),\\
		I_\nu(t,x,n)=0&x_1=s(t),\;n_1>0,\\
	\end{cases}
\end{equation*}
by characteristics. This procedure is similar to the computation in Section 2.2 of \cite{dematte2023diffusion}.\\ 
\begin{minipage}{0.55\textwidth}
	 We define for $ (t,x,n)\in(0,\infty)\times\{\xi\in\R^3:\xi_1>s(t)\}\times\Ss^2 $ the point $ y(t,x,n)\in\{s(t)\}\times\R^2 $ to be the intersection of the half-line starting from $ x $ and moving in direction $ -n $ and we denote $ d(t,x,n) $ to be the distance of $ x $ to the interface $ \{s(t)\}\times\R^2 $ in direction $ -n $. We are hence considering
	 \begin{equation*}
	 	\begin{split}
	 		&y(t,x,n)=\{s(t)\}\times\R^2\cap\{x-tn:t>0\}\quad\text{ and }\\&d(t,x,n)=|x-y(x,n)| \text{ such that } x=y(t,x,n)+d(t,x,n)n.
	 	\end{split}
	 \end{equation*} We also define $ d(t,x,n)=\infty $ if $ n_1<0 $.
 An easy application of trigonometry shows also that \[d(t,x,n)\cdot \cos(\theta(n,e_1))=x_1-s(t),\]where $ \theta(n,e_1) $ is the angle between the unit vectors $ n $ and $ e_1=(1,0,0) $. This implies that \[d(t,x,n)n_1=x_1-s(t) \quad \text{ if }n_1>0.\]
\end{minipage}
\begin{minipage}{0.45\textwidth}
	\begin{figure}[H]
		\includegraphics[height=6cm]{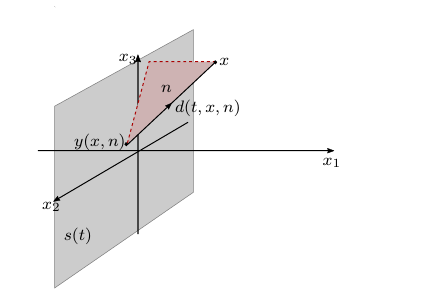}
		\caption{Illustration of the characteristics.}
	\end{figure}
\end{minipage}\\

\noindent Solving the radiative transfer equation by characteristics we hence obtain for $ x_1>s(t) $
\[ I_\nu(t,x,n)=\int_0^{d(t,x,n)}d\tau \alpha \exp\left(-\alpha \tau\right)B_\nu(T(t,x_1-\tau n_1)).\]
As we pointed out above, $ I_\nu $ is not zero on the liquid, i.e. for $ x_1<s(t) $, but is constant to the radiation intensity at the interface. Thus, for $ x_1<s(t) $ we have
\[ I_\nu(t,x,n)=\mathds{1}_{\{n_1\leq0\}}\int_0^\infty d\tau \alpha \exp\left(-\alpha \tau\right)B_\nu(T(t,s(t)-\tau n_1)).\]
This proves also the previous claim about the continuity of the radiative flux through the interface.

While $ \Div\left(\int_0^\infty d\nu\int_{\Ss} dn nI_\nu(t,x,n)\right)=0 $ for $ x_1<s(t) $, a similar computation as in \cite{dematte2023diffusion} shows for $ x_1>s(t) $
\begin{equation*}
	\begin{split}
	\Div\left(\int_0^\infty d\nu\int_{\Ss} dn nI_\nu(t,x,n)\right)=&4\pi\sigma\alpha T^4(t,x_1)-4\pi\sigma\alpha\int_{s(t)}^\infty d\eta \frac{\alpha E_1(\alpha|x_1-\eta|)}{2}T^4(t,\eta),
	\end{split}
\end{equation*}
where $ E_1(x)=\int_{|x|}^\infty \frac{e^{-t}}{t}dt $ is the exponential function. This can be proved using that
\begin{equation*}
\int_{\R^2}d\xi \frac{e^{-\alpha\sqrt{y^2+\xi^2}}}{y^2+\xi^2}=2\pi\int_0^\infty d\rho\; \rho\frac{e^{-\alpha\sqrt{y^2+\rho^2}}}{y^2+\rho^2}=\pi\int_0^\infty dr \frac{e^{-\alpha\sqrt{y^2+r}}}{y^2+r}=2\pi\int_{|y|}^\infty dz \frac{e^{-\alpha z}}{z}=2\pi\int_{\alpha |y|}^\infty dz \frac{e^{- z}}{z}.
\end{equation*}

Therefore, we can write the system \eqref{syst.1} as follows
\begin{equation*}\label{syst.2}
	\begin{cases}
		C_L \partial_t T(t,x_1)=K_L \partial_{x_1}^2 T(t,x_1)& x_1<s(t),\\
		C_S \partial_t T(t,x_1)=K_S \partial_{x_1}^2 T(t,x_1)-4\pi\sigma\alpha I_\alpha[T](t,x_1)& x_1>s(t),\\
		T(t,s(t))=T_M&x_1=s(t),\\
		T(0,x)=T_0(x)& x_1\in\R,\\
		\dot{s}(t)=\frac{1}{L}\left(K_S \partial_{x_1} T(t,s(t)^+)-K_L\partial_{x_1} T(t,s(t)^-)\right),
	\end{cases}
\end{equation*}
where $$ I_\alpha[T](t,x_1)=T^4(t,x_1)-\int_{s(t)}^\infty d\eta \frac{\alpha E_1(\alpha|x_1-\eta|)}{2}T^4(t,\eta) .$$
Notice that we have obtained a system of equations depending only on the space variable $ x_1 $. In order to simplify the reading, we write $ x $ instead of $ x_1 $.\\

Next, we see that we can assume without loss of generality $ C_S=K_S=4\pi\sigma\alpha=1 $. To this end we define $ \tau=\frac{4\pi\sigma\alpha}{C_S}t $ and $ \xi=\sqrt{\frac{4\pi\sigma\alpha}{K_S}}x $. Let us also consider $ T(t,x)=\tilde{T}(\tau,\xi) $ and $ \tilde{s}(\tau)=\sqrt{\frac{4\pi\sigma\alpha}{K_S}}s(t) $. In this way we obtain \[\partial_t T(t,x)=\frac{4\pi\sigma\alpha}{C_S}\partial \tau\tilde{T}(\tau,\xi) \quad \text{ and }\quad \partial^2_x T(t,x)=\frac{4\pi\sigma\alpha}{K_S} \partial^2_\xi \tilde{T}(\tau,\xi).\] Defining $ \tilde{\alpha}=\sqrt{\frac{K_S}{4\pi\sigma\alpha}}\alpha $ we see also that for the radiation term a change of variable gives
\begin{equation*}
	\begin{split}
		I_\alpha[T](t,x)=&T^4(t,x)-\int_{s(t)}^\infty d\eta \frac{\alpha E_1(\alpha|x-\eta|)}{2}T^4(t,\eta)\\= &	\tilde{T}^4(\tau,\xi)-\int_{s(t)}^\infty d\eta \frac{\alpha E_1\left(\alpha\sqrt{\frac{K_S}{4\pi\sigma\alpha}}\left|\sqrt{\frac{4\pi\sigma\alpha}{K_S}}(x-\eta)\right|\right)}{2}\tilde{T}^4\left(\tau,\sqrt{\frac{4\pi\sigma\alpha}{K_S}}\eta\right)\\
		=&	\tilde{T}^4(\tau,\xi))-\int_{\tilde{s}(\tau)}^\infty d\zeta \frac{\tilde{\alpha} E_1\left(\tilde{\alpha}\left|\xi-\zeta\right|\right)}{2}\tilde{T}^4\left(\tau,\zeta\right)=I_{\tilde{\alpha}}[\tilde{T}](\tau,\xi).
	\end{split}
\end{equation*}
We see that
$ \dot s(t)=\partial_t \sqrt{\frac{K_S}{4\pi\sigma\alpha}}\tilde{s}\left(\frac{4\pi\sigma\alpha}{C_S}t\right)=\frac{\sqrt{K_S4\pi\sigma\alpha}}{C_S}\dot{\tilde{s}}(\tau)  $ and also
$$ \frac{1}{L}\left(K_S \partial_{x_1} T(t,s(t)^+)-K_L\partial_{x_1} T(t,s(t)^-)\right)=\frac{K_S}{L}\sqrt{\frac{4\pi\sigma\alpha}{K_S}}\left(\partial_{\xi} \tilde{T}(\tau,\tilde{s}(\tau)^+)-\frac{K_L}{K_S}\partial_{\xi} \tilde{T}(\tau,\tilde{s}(\tau)^-)\right)  $$
Moreover, we define also  $ K=\frac{K_L}{K_S}$, $ C=\frac{C_L}{C_S} $, and finally $ \tilde{L}=\frac{L}{C_S} $. With the change of variable above we obtain under this notation
 \begin{equation}\label{syst.3}
 	\begin{cases}
 		C \partial_\tau \tilde{T}(\tau ,\xi)=K\partial_{\xi}^2 \tilde{T}(\tau,\xi)& \xi<\tilde{s}(\tau),\\
 		\partial_\tau \tilde{T}(\tau,\xi)=\partial_{\xi}^2 \tilde{T}(\tau,\xi)-I_{\tilde{\alpha}}[\tilde{T}](\tau,\xi)& \xi>\tilde{s}(t),\\
 		\tilde{T}(\tau,\tilde{s}(\tau))=T_M&\xi=\tilde{s}(\tau),\\
 		\tilde{T}(0,\xi)=\tilde{T}_0(\xi)& \xi\in\R,\\
 		\dot{\tilde{s}}(\tau)=\frac{1}{\tilde{L}}\left( \partial_{\xi} \tilde{T}(\tau,\tilde{s}(\tau)^+)-K \partial_{\xi} \tilde{T}(\tau,\tilde{s}(\tau)^-)\right).
 	\end{cases}
 \end{equation}
In order to simplify the notation we will write $ \xi=x $, $ t=\tau $, $ \tilde{T}=T $, $ \tilde{s}=s $, $ \tilde{\alpha}=\alpha $, and finally $ \tilde{L}=L $.\\

In the following we will study the problem \eqref{syst.3} in a spatial coordinate system which is at rest. Therefore we now perform a change of variable. To this end we define $ y=x-s(t) $ and we set $ T(t,x)=\tilde{T}(t,x-s(t))=\tilde{T}(t,y) $. The time derivative becomes
\[\partial_t T(t,x)= \partial_t \tilde{T}(t,y)-\dot{s}(t)\partial_y \tilde{T}(t,y) .\] Furthermore, the radiation term $ I[T] $ is
\begin{equation*}
\begin{split}
	I_\alpha[T](t,x)=&T^4(t,x)-\int_{s(t)}^\infty d\eta \frac{\alpha E_1(\alpha|x-\eta|)}{2}T^4(t,\eta)\\= &	\tilde{T}^4(t,x-s(t))-\int_{s(t)}^\infty d\eta \frac{\alpha E_1(\alpha|x-s(t)-(\eta-s(t))|)}{2}\tilde{T}^4(t,\eta-s(t))\\
	=&	\tilde{T}^4(t,x-s(t))-\int_{0}^\infty d\xi \frac{\alpha E_1(\alpha|x-s(t)-\xi|)}{2}\tilde{T}^4(t,\xi)\\
	=&	\tilde{T}^4(t,y)-\int_{0}^\infty d\xi \frac{\alpha E_1(\alpha|y-\xi|)}{2}\tilde{T}^4(t,\xi)=I_\alpha[\tilde{T}](t,y).
\end{split}
\end{equation*}
In order to simplify the notation we will write $ \tilde{T}=T $ and we obtain the following system  
\begin{equation}\label{syst.4}
	\begin{cases}
		\partial_t T(t,y)-\dot{s}(t)\partial_y T(t,y)=\frac{K}{C} \partial_{y}^2 T(t,y)& y<0,\\
		\partial_t T(t,y)-\dot{s}(t)\partial_y T(t,y)=\partial_{y}^2 T(t,y)-I_\alpha[T](t,y)& y>0,\\
		T(t,0)=T_M&y=0,\\
		T(0,y)=T_0(y)& y\in\R,\\
		\dot{s}(t)=\frac{1}{L}\left( \partial_{y} T(t,0^+)-K\partial_{y} T(t,0^-)\right).
	\end{cases}
\end{equation}
The rest of the paper is devoted to the study of the free boundary problem \eqref{syst.1} in its equivalent formulation \eqref{syst.4}.\\
\subsection{Some notation}\label{subs.notation}
Let $ U\subseteq\R $. Throughout this article we will denote by $ C^{k,\beta} (U)$ the space of $ k $-times continuous differentiable functions $ f $ with
\[\Arrowvert f\Arrowvert_{k,\beta}=\max\limits_{0\leq j\leq k}\left(\sup\limits_{U}\left|\partial_x^j f\right| \right)+\sup\limits_{x,y\in U}\frac{\left|\partial^k_x f(x)-\partial_x^kf(y)\right|}{|x-y|^\beta}<\infty.\]
We remark that $ f\in C^{k,\beta}(U) $ has all $ k $ derivatives bounded. 

In a similar way we consider the space $ \mathcal{C}^{n,k}_{t,x}\left((0,\tau)\times U\right) $ to be the space of functions $ f\in C^0\left((0,\tau)\times U\right)  $ with continuous derivatives $ \partial_t^jf\in C^0\left((0,\tau)\times U\right) $ and $ \partial_x^lf\in C^0\left((0,\tau)\times U\right) $ for any $ 0\leq j\leq n $ and $ 0\leq l\leq k $. Notice that the functions and their derivatives are continuous up to the boundary but their norms have not to be bounded.

Moreover, for $ 0\leq a<\tau $ the space $ C^{n+\beta, k+\delta}_{t,x}\left([a,\tau]\times U\right) $ is the space of functions $ f\in\mathcal{C}^{n,k}_{t,x}\left([a,\tau]\times U\right)  $ with \[\sup\limits_{[a,\tau]\times U}\left|\partial_t^j f\right|<\infty \text{ for }0\leq j\leq n \quad\text{ and }\sup\limits_{[a,\tau]times U}\left|\partial_x^l f\right|<\infty \text{ for }0\leq l\leq k\]and with\[\sup\limits_{ t,s\in(a,\tau),\;x\in U}\frac{\left|\partial^n_t f(t,x)-\partial_t^nf(s,x)\right|}{|t-s|^\beta}<\infty\quad\text{ and }\sup\limits_{x,y\in U,\; t\in(a,\tau)}\frac{\left|\partial^k_x f(t,x)-\partial_x^kf(t,y)\right|}{|x-y|^\delta}<\infty.\]
In particular, if $ 2(n+\beta)=k+\delta $ the derivatives of the function $ f\in C^{n+\beta, k+\delta}_{t,x}\left([a,\tau]\times U\right)  $ satisfy also $ \partial_t^j\partial_x^l f \in C^0\left([a,\tau]\times U\right)$ for all $ 2j+l<k $ and $ \beta,\delta\in[0,1) $.

Finally, when we write that the domain of the functions $ v_i $ is $ \R_\pm $ (or $ [0,\pm R] $) we always mean that $ v_1 $ is a function on $ \R_- $ (resp. on $ [-R,0] $) and that $ v_2 $ is a function on $ \R_+ $ (resp. on $ [0,R] $).
\section{Local well-posedness}\label{Sec.loc.Well.Pos}
In this section we prove the local well-posedness theory for the free boundary problem \eqref{syst.4}. Later, in Section \ref{subs.global} we will extend the result for a large class of initial data for which a global well-posedness result will be proved. In the following subsections we will show with a fixed-point argument the existence of a unique solution for small times. Further on we will show the regularity and some properties of the solutions.
\subsection{Fixed-point method}
We show the local well-posedness for the system \eqref{syst.4}. We will moreover denote by $ T_1 $ the temperature defined for  $ t>0 $ and $ y<0 $ and by $ T_2 $ the temperature defined for $ t>0 $ and $ y>0 $. Hence the system can be rewritten as
\begin{equation}\label{syst.5}
	\begin{cases}
		\partial_t T_1(t,y)-\dot{s}(t)\partial_y T_1(t,y)=\frac{K}{C} \partial_{y}^2 T_1(t,y)& y<0,\\
		\partial_t T_2(t,y)-\dot{s}(t)\partial_y T_2(t,y)=\partial_{y}^2 T_2(t,y)-I_\alpha[T_2](t,y)& y>0,\\
		T_1(t,0)=T_2(t,0)=T_M&y=0,\\
		T_1(0,y)=T_0(y) \text{ and } T_2(0,y)=T_0(y)& y\in\R,\\
		\dot{s}(t)=\frac{1}{L}\left( \partial_{y} T_2(t,0^+)-K\partial_{y} T_1(t,0^-)\right)&t>0\\
		s(0)=0.
	\end{cases}
\end{equation}

\begin{theorem}\label{thm.local.well.pos.}
	Let $ T_0\in C_b^{0,1}(\R) $ be bounded and positive with $ T_0(0)=T_M $, $ T_0(y)>T_M $ if $ y<0 $ and $ T_0(y)<T_M $ if $ y>0 $. Let also $ T_0\left.\right|_{\R_\pm}\in C^2(\R_\pm) $ with bounded first and second derivative. Then for a time $ t^*>0 $ small enough there exists a unique bounded solution $ (T_1,T_2, s)\in \mathcal{C}^{0,1}_{t,y}((0,t^*)\times\R_-)\times \mathcal{C}^{0,1}_{t,y}((0,t^*)\times\R_+)\times C^1((0,t^*)) $ solving the problem \eqref{syst.5} in distributional sense.
\end{theorem}
\begin{proof}
We follow the same strategy used by Ruben\v{s}te\u{\i}n in \cite{Rubinstein} and by Friedman in \cite{Friedman5}. We will construct with the help of suitable Green's functions the (implicit) solution formula for $ T_1 $, $ T_2 $ and $ s $ and we will use a contraction argument in order to show the existence of a unique solution.\\

In order to simplify the computations we consider the equivalent problem for $ u_1:=T_1-T_M $ and $ u_2:=T_2-T_M $. Hence we consider the system
\begin{equation}\label{syst.6}
	\begin{cases}
		\partial_t u_1(t,y)-\dot{s}(t)\partial_y u_1(t,y)=\frac{K}{C} \partial_{y}^2 u_1(t,y)& y<0,\\
		\partial_t T_2(t,y)-\dot{s}(t)\partial_y u_2(t,y)=\partial_{y}^2 u_2(t,y)-I_\alpha[u_2+T_M](t,y)& y>0,\\
		u_1(t,0)=u_2(t,0)=0&y=0,\\
		u_1(0,y)=u_0(y) \text{ and } u_2(0,y)=u_0(y)& y\in\R,\\
		\dot{s}(t)=\frac{1}{L}\left( \partial_{y} u_2(t,0^+)-K\partial_{y} u_1(t,0^-)\right)&t>0\\
		s(0)=0,
	\end{cases}
\end{equation}
where $ u_0:=T_0-T_M $ satisfies $ u_0(0)=0 $ and $ u_0>0 $ if $ y<0 $ as well as $ u_0<0 $ if $ y>0 $.

Let\[ G(x,\xi, a(t-\tau)):= \Phi(x-\xi, a(t-\tau))-\Phi(x+\xi,a(t-\tau))\] be the Green's function for the half space $ \R_+ $, where $ \Phi(z,s)=\frac{1}{\sqrt{4\pi z}}\exp\left(-\frac{z^2}{4s}\right) $ is the fundamental solution of the heat equation. Recall that with the help of the Green's function $ G $ the solution of the following Cauchy problem on $ \R_+ $ or on $ \R_- $ 
	\begin{equation*}
		\begin{cases}
			\partial_t v-a\partial_x^2 v=f& x>0 \text{ resp. }x<0,\\
			v(0,t)=0& t>0,\\
			v(x,0)=g_0(x)& x>0 \text{ resp. }x<0
				\end{cases}
	\end{equation*}
has the integral representation
$$ v(t,x)=\int_0^\infty g_0(\xi)G(x,\xi,at)d\xi+ \int_0^\infty \int_0^t f(\xi,\tau)G(x,\xi,a(t-\tau))d\tau d\xi $$ on $ \R_+ $ and $$ v(t,x)=\int_{-\infty}^0 g_0(\xi)G(x,\xi,at)d\xi+ \int_{-\infty}^0 \int_0^t f(\xi,\tau)G(x,\xi,a(t-\tau))d\tau d\xi $$ on $ \R_- $.
Hence, we obtain for $ u_1 $ and $ u_2 $ the following identities considering $ \dot{s}(t)\partial_yu_1 $ resp. $ \dot{s}(t)\partial_yuu_2-I_\alpha[u_2+T_M] $ as sources
\begin{equation}\label{u1}
u_1(t,y)=\int_{-\infty}^0 u_0(\xi)G(y,\xi,\kappa t)d\xi+ \int_{-\infty}^0 \int_0^t \dot{s}(\tau)\partial_\xi u_1(\tau,\xi)G(y,\xi,\kappa(t-\tau))d\tau d\xi,
\end{equation}
where we used the notation $ \kappa=\frac{K}{C} $, and
\begin{multline}\label{u2}
	u_2(t,y)=\int_0^\infty u_0(\xi)G(y,\xi,t)d\xi+ \int_0^\infty \int_0^t \dot{s}(\tau)\partial_\xi u_2(\tau,\xi)G(y,\xi,(t-\tau))d\tau d\xi\\- \int_0^\infty \int_0^t I_\alpha[u_2+T_M](\tau,\xi)G(y,\xi,(t-\tau))d\tau d\xi.
\end{multline}
We now have to differentiate these expressions with respect to the spatial coordinate in order to find an expression for $ \dot{s} $. We recall that \[
\partial_y G(y,\xi,a(t-\tau))=-\partial_\xi \left(\Phi(y-\xi,a(t-\tau))+\Phi(y+\xi,a(t-\tau))\right)=-\partial_\xi g(y,\xi,a(t-\tau)),
\]
where $ g(y,\xi,a(t-\tau))=\Phi(y-\xi,a(t-\tau))+\Phi(y+\xi,a(t-\tau)) $. This implies on one hand 
\begin{equation}\label{du1}
	\partial_yu_1(t,y)=\int_{-\infty}^0 \partial_\xi u_0(\xi)g(y,\xi,\kappa t)d\xi- \int_{-\infty}^0 \int_0^t \dot{s}(\tau)\partial_\xi u_1(\tau,\xi)\partial_\xi g(y,\xi,\kappa(t-\tau))d\tau d\xi,
\end{equation}
where we integrated by parts \[-\int_{-\infty}^0 u_0(\xi) \partial_\xi g(y,\xi,\kappa t)d\xi=\int_{-\infty}^0 \partial_\xi u_0(\xi)g(y,\xi,\kappa t)d\xi\] since $ u_0(0)=0 $ and $ g\to 0 $ as $ |\xi|\to\infty $ for every fixed $ y $.

On the other hand we have also
\begin{multline}\label{du2}
	\partial_y u_2(t,y)=\int_0^\infty \partial_\xi u_0(\xi)g(y,\xi,t)d\xi- \int_0^\infty \int_0^t \dot{s}(\tau)\partial_\xi u_2(\tau,\xi)\partial_\xi g(y,\xi,(t-\tau))d\tau d\xi\\+ \int_0^\infty \int_0^t I_\alpha[u_2+T_M](\tau,\xi)\partial_\xi g(y,\xi,(t-\tau))d\tau d\xi.
\end{multline}
Hence, $ \dot{s}(t) $ is given by
\begin{multline}\label{s}
	\dot{s}(t)=\frac{1}{L}\left(\int_0^\infty \partial_\xi u_0(\xi)g(0^+,\xi,t)d\xi- \int_0^\infty \int_0^t \dot{s}(\tau)\partial_\xi u_2(\tau,\xi)\partial_\xi g(0^+,\xi,(t-\tau))d\tau d\xi\right.\\\left.+ \int_0^\infty \int_0^t I_\alpha[u_2+T_M](\tau,\xi)\partial_\xi g(0^+,\xi,(t-\tau))d\tau d\xi-K\int_{-\infty}^0 \partial_\xi u_0(\xi)g(0^-,\xi,\kappa t)d\xi\right.\\\left.+K \int_{-\infty}^0 \int_0^t \dot{s}(\tau)\partial_\xi u_1(\tau,\xi)\partial_\xi g(0^-,\xi,\kappa(t-\tau))d\tau d\xi\right).
\end{multline}
Equations \eqref{u1}, \eqref{u2} and \eqref{s} define the operator $ \mathcal{L}(u_1,u_2,\dot{s}) $, for which we will show that there exists a unique fixed-point in a suitable set. This will conclude the proof of the existence of a unique solution for small times. Indeed, since $ s(0)=0 $ the solution to the problem \eqref{syst.6} is given by $ \left(u_1,u_2,\int_0^t \dot{s}(\tau)d\tau\right) $. Before defining the space in which we will work and proving the contraction property for the operator $ \mathcal{L} $, we collect some key estimates.\\

First of all since $ \partial_\xi\Phi(y-\xi,a(t-\tau))=\frac{(\xi-y)}{2a(t-\tau)}\frac{1}{\sqrt{ 4\pi a (t-\tau)}} \exp\left(-\frac{|y-\xi|^2}{4a(t-\tau)}\right)$ we estimate
\begin{multline}\label{est.int.dg+}
\int_0^\infty 	\left|\partial_\xi g(y,\xi,a(t-\tau))\right| d\xi\leq \int_0^\infty \frac{|\xi-y|}{2a(t-\tau)}\frac{1}{\sqrt{ 4\pi a (t-\tau)}}\exp\left(-\frac{|y-\xi|^2}{4a(t-\tau)}\right)d\xi\\
+\int_0^\infty \frac{|\xi+y|}{2a(t-\tau)}\frac{1}{\sqrt{ 4\pi a (t-\tau)}}\exp\left(-\frac{|y+\xi|^2}{4a(t-\tau)}\right)d\xi\\
=\int_{-y}^\infty \frac{|\xi|}{2a(t-\tau)}\frac{1}{\sqrt{ 4\pi a (t-\tau)}}\exp\left(-\frac{|\xi|^2}{4a(t-\tau)}\right)d\xi+\int_y^\infty \frac{|\xi|}{2a(t-\tau)}\frac{1}{\sqrt{ 4\pi a (t-\tau)}}\exp\left(-\frac{|\xi|^2}{4a(t-\tau)}\right)d\xi\\
=\int_0^\infty \frac{\xi}{a(t-\tau)}\frac{1}{\sqrt{ 4\pi a (t-\tau)}}\exp\left(-\frac{|\xi|^2}{4a(t-\tau)}\right)d\xi=\frac{1}{\sqrt{ \pi a (t-\tau)}},
\end{multline}
where we used the change of coordinate $ \xi'=\xi-y $ resp. $ \xi'=\xi+y $ and the fact that the resulting function is even. In the very same way we can estimate also 
\begin{equation*}\label{est.int.dg-}
\int_{-\infty}^0 	\left|\partial_\xi g(y,\xi,a(t-\tau))\right| d\xi\leq \frac{1}{\sqrt{ \pi a (t-\tau)}}.
\end{equation*}
A direct consequence of these estimates are the following results
\begin{equation}\label{est.double.int.dg+}
\int_0^\infty \int_0^t	\left|\partial_\xi g(y,\xi,a(t-\tau))\right| d\xi d\tau\leq \int_0^t \frac{1}{\sqrt{ \pi a (t-\tau)}}d\tau=\frac{1}{\sqrt{ a\pi}}\int_0^t\frac{1}{\sqrt{\tau}}d\tau=\frac{2\sqrt{t}}{\sqrt{ a\pi}},
\end{equation}
where we used Fubini's theorem and \eqref{est.int.dg+}. Analogously we also have
\begin{equation}\label{est.double.int.dg-}
	\int_{-\infty}^0 \int_0^t	\left|\partial_\xi g(y,\xi,a(t-\tau))\right| d\xi d\tau\leq \frac{2\sqrt{t}}{\sqrt{ a\pi}}.
\end{equation}
Another important estimates are the following ones
\begin{multline}\label{est.int.G+}
	\int_{0}^\infty \left|G(y,\xi,a(t-\tau))\right|d\xi\leq 	\int_{0}^\infty \Phi(y-\xi,a(t-\tau))+\Phi(y+\xi,a(t-\tau))d\xi\\=2\int_{|y|}^\infty \Phi(\xi,a(t-\tau))d\xi+\int_{-|y|}^{|y|} \Phi(\xi,a(t-\tau))d\xi \\=2\int_0^\infty\Phi(\xi,a(t-\tau))d\xi=\int_{-\infty}^\infty \Phi(\xi,a(t-\tau))d\xi=1 ,
\end{multline}
where we used the change of variables $ \xi'=\xi-y $ resp. $ \xi'=\xi+y $ and the fact that $ \Phi $ is a non-negative even function with integral 1. The same holds on the negative real line \begin{equation}\label{est.int.G-}
	\int_{-\infty}^0 \left|G(y,\xi,a(t-\tau))\right|d\xi\leq 1 .
\end{equation}
Similarly, using the definition of $ g(y,\xi,a(t-\tau)) $ we compute
\begin{equation}\label{est.int.g+}
	\int_{0}^\infty \left|g(y,\xi,a(t-\tau))\right|d\xi\leq \int_{-\infty}^\infty \Phi(\xi,a(t-\tau))d\xi=1
\end{equation}
and 
\begin{equation}\label{est.int.g-}
	\int_{-\infty}^0 \left|g(y,\xi,a(t-\tau))\right|d\xi\leq \int_{-\infty}^\infty \Phi(\xi,a(t-\tau))d\xi=1.
\end{equation}
We define now the metric space for which we will apply the Banach fixed-point theorem. Let us first introduce the notation that we will use throughout this section. 
On $ \mathcal{C}^{0,1}_{t,y}\left((0,t^*)\times\R_\pm\right) $ we consider the norm $ \Arrowvert f\Arrowvert_{0,1}:=\max\{\Arrowvert f\Arrowvert_{C^0},\Arrowvert \partial_yf\Arrowvert_{C^0}\} $. Let $ \theta\in(0,1) $ and let  
\[C_1,\;C_2>\frac{\Arrowvert u_0\Arrowvert_{1}}{1-\theta}\text{ and }C_3>\frac{1+K}{L}\frac{\Arrowvert u_0\Arrowvert_{1}}{1-\theta}.\] We consider the following closed metric space
\begin{multline*}
		\mathcal{A}_{C_1,C_2,C_3}=\left\{(u_1,u_2,\dot{s})\in\mathcal{C}^{0,1}_{t,y}\left((0,t^*)\times\R_-\right)\times\mathcal{C}^{0,1}_{t,y}\left((0,t^*)\times\R_+\right)\times C^0((0,t^*)):\right.\\\; \Arrowvert u_1\Arrowvert_{0,1}\leq C_1,\;\Arrowvert u_2\Arrowvert_{0,1}\leq C_2,\;\Arrowvert \dot{s}\Arrowvert_{C^0}\leq C_3 \bigg\},
\end{multline*}
with the metric induced by the norm $ \Arrowvert (u_1,u_2,\dot{s})\Arrowvert_{\mathcal{A}}:=\Arrowvert u_1\Arrowvert_{0,1}+\Arrowvert u_2\Arrowvert_{0,1}+\Arrowvert \dot{s}\Arrowvert_{C^0} $.
We consider the following operator acting on $ \mathcal{A} $

\begin{center}
	\begin{tabular}{c c c c}
		$ \mathcal{L}: $&$ \mathcal{A}=	\mathcal{A}_{C_1,C_2,C_3} $&$ \to $& $ \mathcal{C}^{0,1}_{t,y}\left((0,t^*)\times\R_-\right)\times\mathcal{C}^{0,1}_{t,y}\left((0,t^*)\times\R_+\right)\times C^0((0,t^*)) $\\ &&&\\
		 &$ (u_1,u_2,\dot{s}) $&$ \mapsto $ & $ \left(\mathcal{L}_1((u_1,u_2,\dot{s})),\mathcal{L}_2((u_1,u_2,\dot{s})),\mathcal{L}_2((u_1,u_2,\dot{s}))\right), $
	\end{tabular}
\end{center}
where we defined according to \eqref{u1}, \eqref{u2} and \eqref{s}
\begin{equation*}
\mathcal{L}_1((u_1,u_2,\dot{s}))(t,y)=\int_{-\infty}^0 u_0(\xi)G(y,\xi,\kappa t)d\xi+ \int_{-\infty}^0 \int_0^t \dot{s}(\tau)\partial_\xi u_1(\tau,\xi)G(y,\xi,\kappa(t-\tau))d\tau d\xi,
\end{equation*}
\begin{multline*}
\mathcal{L}_2((u_1,u_2,\dot{s}))(t,y)=\int_0^\infty u_0(\xi)G(y,\xi,t)d\xi+ \int_0^\infty \int_0^t \dot{s}(\tau)\partial_\xi u_2(\tau,\xi)G(y,\xi,(t-\tau))d\tau d\xi\\- \int_0^\infty \int_0^t I_\alpha[u_2+T_M](\tau,\xi)G(y,\xi,(t-\tau))d\tau d\xi
\end{multline*}
and 
\begin{multline*}
\mathcal{L}_3((u_1,u_2,\dot{s}))(t)=\frac{1}{L}\left(\int_0^\infty \partial_\xi u_0(\xi)g(0^+,\xi,t)d\xi- \int_0^\infty \int_0^t \dot{s}(\tau)\partial_\xi u_2(\tau,\xi)\partial_\xi g(0^+,\xi,(t-\tau))d\tau d\xi\right.\\\left.+ \int_0^\infty \int_0^t I_\alpha[u_2+T_M](\tau,\xi)\partial_\xi g(0^+,\xi,(t-\tau))d\tau d\xi-K\int_{-\infty}^0 \partial_\xi u_0(\xi)g(0^-,\xi,\kappa t)d\xi\right.\\\left.+K \int_{-\infty}^0 \int_0^t \dot{s}(\tau)\partial_\xi u_1(\tau,\xi)\partial_\xi g(0^-,\xi,\kappa(t-\tau))d\tau d\xi\right).
\end{multline*}
With the help of \eqref{du1} and \eqref{du2} and using the properties of the Green's functions, it is not difficult to see that the operator is well-defined. We prove next that choosing a time $ t^*>0 $ small enough and for $ C_1,\;C_2,\;C_3 $ large enough the operator is a self-map. To this end, we need to estimate the norms of the three components of $ \mathcal{L} $.

We assumed \[\Arrowvert u_0\Arrowvert_{1}=\max\left\{\Arrowvert u_0\Arrowvert_{C^0},\Arrowvert \partial_yu_0\Arrowvert_{C^0(\R_-)},\Arrowvert \partial_yu_0\Arrowvert_{C^0(\R_+)}\right\}<\infty. \] Recall that $ \alpha>0 $, $ K>0 $, $ L>0 $ and $ \kappa>0$ are all constants. Combining the triangle inequality as well as H\"older's inequality with the estimate \eqref{est.int.G-} we can conclude
\begin{equation}\label{L1C0}
\begin{split}
		\Arrowvert \mathcal{L}_1((u_1,u_2,\dot{s}))\Arrowvert_{C^0}\leq& \Arrowvert u_0\Arrowvert_{C^0} \sup\limits_{0<t\leq t^*}\int_{-\infty}^0|G(y,\xi,\kappa t)|d\xi+\Arrowvert u_1\Arrowvert_{0,1}	\Arrowvert \dot{s}\Arrowvert_{C^0}\sup\limits_{0<t\leq T}\int_{-\infty}^0\int_0^t |G(y,\xi,\kappa(t-\tau))|d\tau d\xi\\
		\leq &  \Arrowvert u_0\Arrowvert_{1}+t^*\Arrowvert u_1\Arrowvert_{0,1}	\Arrowvert \dot{s}\Arrowvert_{C^0}.
\end{split}
\end{equation}
Before moving to the estimate for the second component of the operator $ \mathcal{L} $ we have to consider the radiation term $ I_\alpha $. Using $ \int_{-\infty}^\infty \frac{E_1(\xi)}{2}d\xi=1 $ and $ |a+b|^4\leq 8|a|^4+8|b|^4 $ we obtain
\begin{equation}\label{IC0}
	\begin{split}
		\Arrowvert I_\alpha[u_2+T_M]\Arrowvert_{C^0}=&\sup\limits_{0<t\leq t^*, y>0}\left| \int_0^\infty \frac{\alpha E_1(\alpha(y-\eta))}{2} (u_2(t,\eta)+T_M)^4d\eta-(u_2(t,y)+T_M)^4\right|\\
		\leq &  \Arrowvert (u_2+T_M)^4\Arrowvert_{C^0}\sup\limits_{y>0}\left(\int_0^\infty \frac{\alpha E_1(\alpha(y-\eta))}{2}d\eta+1\right)\\
		\leq &16\left(\Arrowvert u_2\Arrowvert_{}^4+T_M^4\right).
	\end{split}
\end{equation}
Hence, using now \eqref{est.int.G+} we obtain similarly as in \eqref{L1C0}
\begin{equation}\label{L2C0}
		\Arrowvert \mathcal{L}_2((u_1,u_2,\dot{s}))\Arrowvert_{C^0}
		\leq  \Arrowvert u_0\Arrowvert_{1}+t^*\left(\Arrowvert u_2\Arrowvert_{0,1}	\Arrowvert \dot{s}\Arrowvert_{C^0}+16\Arrowvert u_2\Arrowvert_{0,1}^4+16T_M^4 \right).
\end{equation}
We now estimate the norm of the derivative of the first two component of $ \mathcal{L} $. Notice that $ \partial_y\mathcal{L}_1(u_1,u_2,\dot) $ is given by the right hand side of \eqref{du1}, while $\partial_y\mathcal{L}_2(u_1,u_2,\dot) $ is given by the right hand side of \eqref{du2}. Hence, using this time \eqref{est.int.g-} and \eqref{est.double.int.dg-}, we obtain in a similar manner as \eqref{L1C0}
\begin{equation}\label{L1C1}
	\Arrowvert \partial_y \mathcal{L}_1((u_1,u_2,\dot{s}))\Arrowvert_{C^0}
	\leq  \Arrowvert u_0\Arrowvert_{1}+\frac{2}{\sqrt{\kappa \pi}}\sqrt{t^*}\Arrowvert u_1\Arrowvert_{0,1}	\Arrowvert \dot{s}\Arrowvert_{C^0}.
\end{equation}
Analogously, \eqref{est.int.g+}, \eqref{est.double.int.dg+} and \eqref{IC0} imply
\begin{equation}\label{L2C1}
	\Arrowvert \partial_y \mathcal{L}_2((u_1,u_2,\dot{s}))\Arrowvert_{C^0}
	\leq  \Arrowvert u_0\Arrowvert_{1}+\frac{2}{\sqrt{ \pi}}\sqrt{t^*}\left(\Arrowvert u_2\Arrowvert_{0,1}	\Arrowvert \dot{s}\Arrowvert_{C^0}+16\Arrowvert u_2\Arrowvert_{0,1}^4+16T_M^4 \right).
\end{equation}
Finally, combining \eqref{L1C1} and \eqref{L2C1} we have
\begin{equation}\label{L3C0}
	\Arrowvert \mathcal{L}_3((u_1,u_2,\dot{s}))\Arrowvert_{C^0}
	\leq \frac{1}{L}\left( (1+K) \Arrowvert u_0\Arrowvert_{1}+\frac{2}{\sqrt{ \pi}}\sqrt{t^*}\left(\Arrowvert u_2\Arrowvert_{0,1}	\Arrowvert \dot{s}\Arrowvert_{C^0} +\frac{K}{\sqrt{\kappa}}\Arrowvert u_1\Arrowvert_{0,1}	\Arrowvert \dot{s}\Arrowvert_{C^0}+16\Arrowvert u_2\Arrowvert_{0,1}^4+16T_M^4 \right)\right).
\end{equation}
Therefore, for $ (u_1,u_2,\dot{s})\in\mathcal{A} $ combining \eqref{L1C0}, \eqref{L2C0}, \eqref{L1C1}, \eqref{L2C1} and \eqref{L3C0} we obtain
\begin{equation*}\label{L101}
	\Arrowvert  \mathcal{L}_1((u_1,u_2,\dot{s}))\Arrowvert_{0,1}
	\leq  \Arrowvert u_0\Arrowvert_{1}+\left(t^*+\frac{2}{\sqrt{\kappa \pi}}\sqrt{t^*}\right)C_1C_3,
\end{equation*}
\begin{equation*}\label{L201}
	\Arrowvert  \mathcal{L}_2((u_1,u_2,\dot{s}))\Arrowvert_{0,1}
	\leq  \Arrowvert u_0\Arrowvert_{1}+\left(t^*+\frac{2}{\sqrt{ \pi}}\sqrt{t^*}\right)\left(C_2C_3+16C_2^4+16T_M^4\right),
\end{equation*}
and finally 
\begin{equation*}\label{L301}
	\Arrowvert  \mathcal{L}_3((u_1,u_2,\dot{s}))\Arrowvert_{C^0}
	\leq  \frac{1}{L}\left( (1+K) \Arrowvert u_0\Arrowvert_{1}+\frac{2}{\sqrt{ \pi}}\sqrt{t^*}\left(C_2C_3 +\frac{K}{\sqrt{\kappa}}C_1C_3+16C_2^4+16T_M^4 \right)\right).
\end{equation*}
Then defining
\[ t_1=\frac{1}{2}\min\left\{\frac{\theta}{C_3}, \frac{\theta^2 \kappa \pi}{8 C_3^2}\right\},\quad t_2=\frac{1}{6}\min\left\{\frac{\theta}{C_3}, \frac{\theta}{16C_2^3},\frac{\theta C_2}{16 T_M^4},\frac{1}{6}\left(\frac{\sqrt{\pi}\theta}{2C_3}\right)^2,\frac{1}{6}\left(\frac{\sqrt{\pi}\theta}{32C_2^3}\right)^2,\frac{1}{6}\left(\frac{\sqrt{\pi}\theta C_2}{32T_M^4}\right)^2\right\} \] and \[t_3=\frac{L^2\pi\theta^2}{64}\min\left\{\left(\frac{1}{C_2}\right)^2,\left(\frac{\sqrt{\kappa}}{KC_1}\right)^2,\left(\frac{C_3}{16 C_2^4}\right)^2,\left(\frac{C_3}{16 T_M^4}\right)^2\right\}\]
we conclude for $ t^*\leq \min\{t_1,t_2,t_3\} $ that\[ \Arrowvert  \mathcal{L}_1((u_1,u_2,\dot{s}))\Arrowvert_{0,1}\leq C_1,\quad\Arrowvert  \mathcal{L}_2((u_1,u_2,\dot{s}))\Arrowvert_{0,1}\leq C_2\quad \text{and}\quad \Arrowvert  \mathcal{L}_3((u_1,u_2,\dot{s}))\Arrowvert_{C^0}\leq C_3,\] and hence $ \mathcal{L} $ maps $ \mathcal{A} $ into itself.
We show now that for $t^*>0 $ small enough $ \mathcal{L} $ is also a contraction. To this end we assume $ (u_1,u_2,\dot{s}),(\bar u_1,\bar u_2,\dot{\bar s})\in\nolinebreak\mathcal{A} $. First of all we consider the radiation term. Using that $ a^4-b^4=(a-b)(a^3+a^2b+ab^2+b^3)=(a-b)p_3(a,b) $ and that $ p_3(a,b)\leq 2(|a|^3+|b|^3) $ we estimate
\begin{equation}\label{I-IC0}
	\begin{split}
		\Arrowvert& I_\alpha[u_2+T_M]-I_\alpha[\bar u_2+T_M]\Arrowvert_{C^0}\\=&\sup\limits_{0<t\leq t^*, y>0}\left| \int_0^\infty \frac{\alpha E_1(\alpha(y-\eta))}{2} (u_2(t,\eta)-\bar{u}_2(t,\eta))p_3(u_2+T_M,\bar{u}_2+T_M)(t,\eta)d\eta\right.\\&\quad\quad\quad\quad\quad-(u_2(t,y)-\bar{u}_2(t,y))p_3(u_2+T_M,\bar{u}_2+T_M)(t,y)\bigg|\\
		\leq&  2 \Arrowvert u_2-\bar{u}_2\Arrowvert_{0,1}\left(\Arrowvert |u_2+T_M|^3\Arrowvert_{0,1}+\Arrowvert |\bar u_2+T_M|^3\Arrowvert_{0,1}\right)\sup\limits_{y>0}\left(\int_0^\infty \frac{\alpha E_1(\alpha(y-\eta))}{2}d\eta+1\right)\\
		\leq &32\left(C_2^3+T_M^3\right)\Arrowvert u_2-\bar{u}_2\Arrowvert_{0,1}.
	\end{split}
\end{equation}
Hence, using triangle inequality as well as \eqref{est.int.G-} and \eqref{est.double.int.dg-} we see
\begin{equation*}\label{L1-L101}
	\begin{split}
		\Arrowvert \mathcal{L}_1((u_1,u_2,\dot{s}))-&\mathcal{L}_1((\bar u_1,\bar u_2,\dot{\bar s}))\Arrowvert_{0,1}\leq\sup\limits_{0<t\leq t^*,y>0}\left|\int_{-\infty}^0 \int_0^t \left(\dot{s}(\tau)\partial_\xi u_1(\tau,\xi)-\dot{\bar s}(\tau)\partial_\xi u_1(\bar \tau,\xi)\right)G(y,\xi,\kappa(t-\tau))d\tau d\xi\right|\\
		&+\sup\limits_{0<t\leq t^*,y>0}\left|\int_{-\infty}^0 \int_0^t \left(\dot{s}(\tau)\partial_\xi u_1(\tau,\xi)-\dot{\bar s}(\tau)\partial_\xi u_1(\bar \tau,\xi)\right)\partial_\xi g(y,\xi,\kappa(t-\tau))d\tau d\xi\right|\\\leq& \left(t^*+\frac{2}{\sqrt{\kappa\pi}}\sqrt{t^*}\right)\left(C_1\Arrowvert \dot{s}-\dot{\bar s}\Arrowvert_{C^0}+C_3 \Arrowvert u_1-\bar u_1\Arrowvert_{0,1}\right) .
	\end{split}
\end{equation*}
Similarly, using \eqref{I-IC0}, \eqref{est.int.G+} and \eqref{est.double.int.dg+} we have
\begin{equation*}\label{L2-L201}
	\begin{split}
		\Arrowvert \mathcal{L}_2((u_1,u_2,\dot{s}))-\mathcal{L}_2((\bar u_1,\bar u_2,\dot{\bar s}))\Arrowvert_{0,1}\leq& \left(t^*+\frac{2}{\sqrt{\pi}}\sqrt{t^*}\right)\left(C_2\Arrowvert \dot{s}-\dot{\bar s}\Arrowvert_{C^0}+C_3 \Arrowvert u_2-\bar u_2\Arrowvert_{0,1}\right)\\&+32(C_2^3+T_M^3)\left(t^*+\frac{2}{\sqrt{\pi}}\sqrt{v}\right)\Arrowvert u_2-\bar u_2\Arrowvert_{0,1} .
	\end{split}
\end{equation*}
Finally, combining the results for the derivatives we obtain
\begin{equation*}\label{L3-L3C0}
	\begin{split}
		\Arrowvert \mathcal{L}_3((u_1,u_2,\dot{s}))-\mathcal{L}_3((\bar u_1,\bar u_2,\dot{\bar s}))\Arrowvert_{C^0}\leq&\frac{2}{\sqrt{\kappa\pi}L}\sqrt{t^*}\left(C_1\Arrowvert \dot{s}-\dot{\bar s}\Arrowvert_{C^0}+C_3 \Arrowvert u_1-\bar u_1\Arrowvert_{0,1}\right)\\+& \frac{2}{\sqrt{\pi}L}\sqrt{t^*}\left(C_2\Arrowvert \dot{s}-\dot{\bar s}\Arrowvert_{C^0}+C_3 \Arrowvert u_2-\bar u_2\Arrowvert_{0,1}+32(C_2^3+T_M^3)\Arrowvert u_2-\bar u_2\Arrowvert_{0,1}\right).
	\end{split}
\end{equation*}
Hence, we have
\begin{equation*}\label{L-L}
	\begin{split}
		\Arrowvert \mathcal{L}((u_1,u_2,\dot{s}))-\mathcal{L}((\bar u_1,\bar u_2,\dot{\bar s}))\Arrowvert_{\mathcal{A}}\leq&\left(t^*+\frac{2}{\sqrt{\kappa\pi}}\sqrt{t^*}+\frac{2}{\sqrt{\kappa\pi}L}\sqrt{t^*}\right)C_3\Arrowvert u_1-\bar u_1\Arrowvert_{0,1}\\
		+&\left(t^*+\frac{2}{\sqrt{\pi}}\sqrt{t^*}+\frac{2}{\sqrt{\pi}L}\sqrt{t^*}\right)\left(C_3+32(C_2^3+T_M^3)\right)\Arrowvert u_2-\bar u_2\Arrowvert_{0,1}\\
		+\bigg[\bigg(t^*+&\left.\left.\frac{2}{\sqrt{\kappa\pi}}\sqrt{t^*}+\frac{2}{\sqrt{\kappa\pi}L}\sqrt{t^*}\right)C_1+\left(t^*+\frac{2}{\sqrt{\pi}}\sqrt{t^*}+\frac{2}{\sqrt{\pi}L}\sqrt{t^*}\right)C_2\right]\Arrowvert \dot{s}-\dot{\bar s}\Arrowvert_{C^0}.
	\end{split}
\end{equation*}
Let $ \lambda\in(0,1) $. We take
\begin{multline*}\label{T}
	t^*=\min\bigg\{t_1,t_2,t_3,\frac{\lambda}{3C_3}, \left(\frac{\lambda\sqrt{\kappa\pi}}{6C_3}\right)^2,\left(\frac{\lambda\sqrt{\kappa\pi}L}{6C_3}\right)^2 ,\frac{\lambda}{3\left(C_3+32(C_2^3+T_M^3)\right)},\left(\frac{\lambda\sqrt{\pi}}{6\left(C_3+32(C_2^3+T_M^3)\right)}\right)^2,\\\\\left(\frac{\lambda\sqrt{\pi}L}{6\left(C_3+32(C_2^3+T_M^3)\right)}\right)^2,\frac{\lambda}{3(C_1+C_2)}, \left(\frac{\lambda\sqrt{\pi}}{6(C_1\kappa^{-1}+C_2)}\right)^2,\left(\frac{\lambda\sqrt{\pi}L}{6(C_1\kappa^{-1}+C_2)}\right)^2 \bigg\}.
\end{multline*}
It is now easy to see that $ \mathcal{L} $ is a contraction with \[ 	\Arrowvert \mathcal{L}((u_1,u_2,\dot{s}))-\mathcal{L}((\bar u_1,\bar u_2,\dot{\bar s}))\Arrowvert_{\mathcal{A}}\leq\lambda	\Arrowvert (u_1,u_2,\dot{s})-(\bar u_1,\bar u_2,\dot{\bar s})\Arrowvert_{\mathcal{A}} .\] Thus, an application of Banach fixed-point theorem yields the existence of a unique solution $ (u_1,u_2,\dot s)\in\mathcal{A}_{C_1,C_2,C_3} $ to the fixed-point system defined by the equations \eqref{u1}, \eqref{u2} and \eqref{s}. Moreover, this is the unique solution in the sense that if $ (\bar u_1,\bar u_2,\dot {\bar s})\in\mathcal{A}_{\bar C_1,\bar C_2,\bar C_3} $ is another fixed-point solution on $ [0,\bar t^*] $ for $ (C_1,C_2,C_3)\ne (\bar C_1,\bar C_2,\bar C_3) $, then $ (u_1,u_2,\dot s)=(\bar u_1,\bar u_2,\dot {\bar s}) $ for all $ t\leq\{t^*,\bar{t}^*\} $. We notice that by construction $ \left(u_1,u_2,\int_0^t\dot s(\tau)d\tau\right) $ solves the problem \eqref{syst.6} in distributional sense. Notice that $ \dot{s} $ satisfies the Stefan condition strongly.
Thus, the theorem is proved.
\end{proof}
In the next subsection we prove that the unique distributional solution found in Theorem \ref{thm.local.well.pos.} is also a classical solution with H\"older regularity.
\subsection{Regularity}
In this section we show that the local distributional solutions $ u_1\in \mathcal{C}_{t,y}^{0,1}\left([0,T]\times \R_-\right) $ and $ u_2\in \mathcal{C}_{t,y}^{0,1}\left([0,T]\times \R_+\right) $ of the parabolic equations 
\begin{equation}\label{parabolic.eq}
	\begin{cases}
		\partial_t u_1(t,y)-\dot{s}(t)\partial_y u_1(t,y)=\frac{K}{C} \partial_{y}^2 u_1(t,y)& y<0,\\
		\partial_t u_2(t,y)-\dot{s}(t)\partial_y u_2(t,y)=\partial_{y}^2 u_2(t,y)-I_\alpha[u_2+T_M](t,y)& y>0,\\
		u_1(t,0)=u_2(t,0)=0&y=0,\\
		u_1(0,y)=u_0(y) \text{ and } u_2(0,y)=u_0(y)& y\in\R.
	\end{cases}
\end{equation}
are strong solutions, i.e. $ u_1\in \mathcal{C}_{t,y}^{1,2}\left((0,t^*]\times \R_-\right) $ and $ u_1\in \mathcal{C}_{t,y}^{1,2}\left((0,t^*]\times \R_+\right) $. Moreover, we will show that these solutions have also locally H\"older regularity in the sense that $  u_i\in C^{1+\delta/2,2+\delta}\left([\eps,t^*]\times \R_\pm\right)  $ for any $ \eps>0 $ and for some $ \delta\in (0,1) $. To this end, we will use the following classical result for the heat equation (cf. \cite{Ladyzenskaja} p.273).
\begin{prop}\label{reg.heat}
Let $ \Phi(x,t)=\frac{1}{\sqrt{4\pi t}}e^{-\frac{x^2}{4t}}\mathds{1}_{t>0} $ be the fundamental solution of the heat equation. Then for $ F\in \mathcal{C}_{t,y}^{\delta/2,\delta}\left([0,t^*]\times \R\right) $, $ \varphi\in C^{2,\delta}(\R) $ and $ g\in C^{1,\delta/2}([-\infty,t^*]) $ the following estimates are true
\begin{equation*}\label{F}
	\Arrowvert \Phi* F\Arrowvert_{1+\delta/2,2+\delta}\leq c \Arrowvert  F\Arrowvert_{\delta/2,\delta},
\end{equation*}
\begin{equation}\label{varphi}
	\Arrowvert \Phi*_y \varphi\Arrowvert_{1+\delta/2,2+\delta}\leq c \Arrowvert \varphi\Arrowvert_{\delta}
\end{equation}
and \begin{equation}\label{g}
	\Arrowvert \partial_y\Phi*_t g\Arrowvert_{1+\delta/2,2+\delta}\leq c \Arrowvert g\Arrowvert_{1+\delta},
\end{equation}
where in \eqref{varphi} we mean
\[\Phi*_y \varphi(t,y)=\int_\R \varphi(\xi)\Phi(y-\xi,t)d\xi\]
and in \eqref{g}\[\partial_y\Phi*_t g(t,y)=\int_{-\infty}^t g(\tau)\partial_y\Phi(y,t-\tau)d\tau.\]
Moreover, if $  g\in C^{1,\delta/2}([0,t^*]) $, then $ \partial_y\Phi*_t g\in \mathcal{C}_{t,y}^{1+\delta/2,2+\delta}\left([\eps,t^*]\times \R\right) $ for all $ \eps>0 $.
\end{prop}
We remark that the last statement can be proved following the estimates in \cite{Ladyzenskaja}.
The proof of the regularity of the solutions $ u_1 $ and $ u_2 $ follows from classical parabolic theory. Nevertheless, we recall the key estimates that we will use.\\

\begin{lemma}\label{bnd.value.estimates}
	Let $ \varphi\in C^{1,1}(\R) $. Then
	\begin{equation}\label{bnd.value.est.1}
		\varphi *_y\Phi(\cdot,t)\in \mathcal{C}_{t,y}^{1/2,1+1}([0,t^*]\times \R) 
	\end{equation}
with \[[\varphi *_y\Phi]_{t,1/2}\leq \Arrowvert\varphi'\Arrowvert_\infty\text{ and }[\varphi' *_y\Phi]_{y,\text{Lip}}\leq \Arrowvert\varphi'\Arrowvert_\infty. \]
Moreover, 	\[\varphi *_y\Phi(\cdot,t)\in \mathcal{C}_{t,y}^{1+1/2,2+1}([\eps,t^*]\times \R)\] 
for any $ \eps>0. $
\begin{proof}
	We see that by a change of variables $ \eta=\frac{\xi-y}{\sqrt{t}} $ we obtain
	\[\varphi *_y\Phi(y,t)=\int_\R \varphi(y+\sqrt{t}\eta)\frac{e^{-\eta^2/4}}{\sqrt{4\pi}}d\eta.\]
	A simple computations shows \eqref{bnd.value.est.1} as well as the estimates for the H\"older seminorms. We remark that $ |\sqrt{a}-\sqrt{b}|\leq |a-b|^{1/2} $ for all $ a,b>0 $.
	In  a similar way we can also see that for any $ \eps>0 $ the function $ \varphi *_y\Phi $ has the claimed higher H\"older regularity on $ [\eps,t^*]\times \R $. For example, we see integrating by parts that
	\[\left|\partial_y^2\int_\R \varphi(\xi)\Phi(y-\xi,t)d\xi\right|=\left|\partial_y\int_\R \varphi'(\xi)\Phi(y-\xi,t)d\xi\right|=\left|\int_\R \varphi'(y+\sqrt{t}\eta)\frac{\eta e^{-\eta^2/4}}{\sqrt{4\pi t}}d\xi\right|\leq\frac{\sqrt{2}\Arrowvert\varphi'\Arrowvert_\infty}{\sqrt{\eps}}. \]
	In a similar way we can prove the estimates for $ \partial_t \varphi *_y\Phi  $ as well as the one for the H\"older seminorms.
\end{proof}
\end{lemma}
Further we will use also the following result
\begin{lemma}\label{interior.estimate}
	Let $ G(y,\xi,at)=\Phi(y-\xi,t)-\Phi(y+\xi,at) $ be the fundamental solution of the heat equation in the half-space. Let also $ f_\pm\in C^{0,0}_{t,y}\left([0,t^*]\times \R_\pm\right) $. Then 
	\begin{equation*}\label{interior.hoelder.1}
	f_\pm*G\in \mathcal{C}^{\alpha/2,1+\beta}_{t,y}\left([0,t^*]\times \R_\pm\right)
	\end{equation*} 
for any $ \alpha,\;\beta \in(0,1) $, where we define
\[f_\pm*G(t,y)=\int_0^t \int_{\R_\pm} f_\pm(\tau,\xi)G(y,\xi,a(t-\tau))d\xi d\tau.\]
Moreover, the norm of  $ f_\pm*G $ in the space $ \mathcal{C}^{\alpha/2,1+\beta}_{t,y}\left([0,t^*]\times \R_\pm\right) $ depends only on $ \Arrowvert f\Arrowvert_\infty$, $ t^* $, $ \alpha,\beta $.
\begin{proof}
	It is enough to prove this Lemma only for $ f=f_+\in \mathcal{C}^{0,0}_{t,y}\left([0,t^*]\times \R_+\right) $ and $ a=1 $. Using that $ -\partial_\xi g(y,\xi,t)=\partial_yG(y,\xi,t) $ and equation \eqref{est.double.int.dg+} we have already seen that
	$ f*G\in \mathcal{C}^{0,1}_{t,y}\left([0,t^*]\times \R_+\right) $ with norm bounded by $\Arrowvert f*G\Arrowvert_{0,1}\leq C\max\{t^*,\sqrt{t*}\}\Arrowvert f\Arrowvert_\infty $. Now we only need to show the H\"older regularity of $ f*G $. 
	Let hence $ 0<s<t<t^* $. Since if $ s<t-s $, then $ |t|+|s|\leq 3|t-s| $ so that $ |f*G(t,y)-f*G(s,y)|\leq 3|t-s|\Arrowvert f\Arrowvert_\infty $ by \eqref{est.int.G+}, it is enough to consider $ s>t-s $. 
	\begin{equation*}
	\begin{split}
		\left|\int_0^t \right.&\left.\int_{\R_+} f(\tau,\xi)\Phi(y\pm\xi,t-\tau)d\xi d\tau-\int_0^s \int_{\R_+} f(\tau,\xi)\Phi(y\pm\xi,s-\tau)d\xi d\tau \right|\\\leq&\Arrowvert f\Arrowvert_\infty
	\int_{s-(t-s)}^t\int_{\R_+}	\left|\Phi(y\pm\xi,t-\tau)\right|d\xi d\tau+\Arrowvert f\Arrowvert_\infty
	\int_{s-(t-s)}^s\int_{\R_+}	\left|\Phi(y\pm\xi,s-\tau)\right|d\xi d\tau\\&+\Arrowvert f\Arrowvert_\infty\int_0^s\int_{\R_+} \left|\Phi(y\pm\xi,t-\tau)-\Phi(y\pm\xi,s-\tau)\right|d\xi d\tau
	\\\leq& 3\Arrowvert f\Arrowvert_\infty|t-s|+\Arrowvert f\Arrowvert_\infty\int_0^{s-(t-s)}\int_{\R_+}\Phi(y\pm\xi,t-\tau)\frac{|y\pm\xi|^2}{4(t-\tau)}\frac{t-s}{s-\tau}+\Phi(y\pm\xi,s-\tau)\frac{\sqrt{t-s}}{\sqrt{t-\tau}}d\xi d\tau\\
	\leq& \Arrowvert f\Arrowvert_\infty \left(4+\sqrt{2}\right)\sqrt{t^*}|t-s|^{1/2},
	\end{split}
	\end{equation*}
where we used that $ t-s\leq s-\tau $ if $ \tau<s-(t-s) $. We turn to the H\"older continuity of the spatial derivative. Let us consider $ 0\leq x<y $. Since $ \partial_y f*G $ is uniformly bounded in $ [0,t^*]\times \R_+ $, we have only to show the H\"older condition for $ |x-y|<1 $. Let hence $ |x-y|<1 $. If $ t\leq|y-x|^2 $ then by \eqref{est.double.int.dg+} we see that
\[	\left|\int_0^t \int_{\R_+} f(\tau,\xi)\left(\partial_\xi g(y,\xi,t-\tau)-\partial_\xi g(x,\xi,t-\tau)\right)d\xi d\tau t\right|\leq C\Arrowvert f\Arrowvert_\infty \sqrt{t}\leq C\Arrowvert f\Arrowvert_\infty |x-y|. \]
Let now $ t>|x-y|^2 $. Using that if $ |y\pm\xi|>|x\pm\xi| $, the following estimate holds
\begin{equation*}
	\begin{split}
	\left|\partial_\xi \Phi(y\pm \xi,\tau)-\partial_\xi \Phi(x\pm\xi,\tau)\right|\leq & \frac{|x-y|}{2\tau}\Phi(y\pm \xi,\tau)+|x-y|\frac{|x\pm\xi}{2\tau}\Phi(x\pm\xi,\tau)\frac{|y\pm\xi|+|x\pm\xi|}{4\tau}\\
	\leq &\frac{|x-y|}{2\tau}\Phi(y\pm \xi,\tau)+\Phi\left(\frac{x\pm\xi}{\sqrt{2}},\tau\right)\left(\frac{|x-y|^2}{4\tau^{3/2}}+\frac{|x-y|}{\tau}\right).
	\end{split}
\end{equation*}
We estimate for any $ \beta\in(0,1) $
\begin{equation*}
	\begin{split}
		\int_0^t\int_{\R_+}& |f(t-\tau,\xi)|
		\left|\partial_\xi \Phi(y\pm \xi,\tau)-\partial_\xi \Phi(x\pm\xi,\tau)\right|d\xi d\tau\\\leq&\int_0^{|x-y|^2}\int_{\R_+} |f(t-\tau,\xi)|
		\left|\partial_\xi \Phi(y\pm \xi,\tau)-\partial_\xi \Phi(x\pm\xi,\tau)\right|d\xi d\tau\\&+\int_{|x-y|^2}^t\int_{\R_+} |f(t-\tau,\xi)|
		\left|\partial_\xi \Phi(y\pm \xi,\tau)-\partial_\xi \Phi(x\pm\xi,\tau)\right|d\xi d\tau\\ \leq&2\Arrowvert f\Arrowvert_\infty\int_0^{|x-y|^2} \int_{\R}\left|\partial_\xi \Phi( \xi,\tau)\right|d\xi d\tau+2C\Arrowvert f\Arrowvert_\infty|x-y|\int_{|x-y|^2}^t \frac{1}{\tau}d\tau\\\leq &
		2\Arrowvert f\Arrowvert_\infty|x-y|+2C\Arrowvert f\Arrowvert_\infty|x-y|^{\beta}\int_{|x-y|^2}^t \tau^{-\frac{1+\beta}{2}}d\tau
		\leq \tilde{C} \frac{\Arrowvert f\Arrowvert_\infty }{1-\beta} \max\left\{1,(t^*)^{(1-\beta)/2}\right\}|x-y|^\beta,
	\end{split}
\end{equation*}
where we also used that $ |x-y|<1 $. 
\end{proof}
\end{lemma}
Finally, we will also use the following result, which can be found in \cite{Ladyzenskaja}.
\begin{prop}\label{hoelder}
	Let $ E\in\left\{\R,\R_\pm\right\} $. Let $ u\in \mathcal{C}^{\alpha,1+\beta}_{t,y}\left([0,t^*]\times E\right) $. Then $ \partial_y u $ is $ \frac{\alpha\beta}{1-\beta} $-H\"older in time.
	\begin{proof}
		We refer to Lemma 3.1, Chapter II of \cite{Ladyzenskaja}. This proposition can be proved in an easier way using that for any $ x,y\in E $ and any $ t,s\in[0,t^*] $ we have the following estimates
		\[u(t,y)=u(t,x)\partial_x u(t,x)(y-x)+\mathcal{O}(|x-y|^{1+\beta}) \text{ for any }t\in[0,t^*]\] and \[|u(t,x)-u(s,x)|\leq C |s-t|^\alpha \text{ for any }x\in E.\]
		Thus,we conclude
		\[|\partial_x u(t,x)-\partial_x u(s,x)|\leq C_1\frac{|t-s|^\alpha}{|x-y|}+C_2 |x-y|^\beta\leq C |t-s|^{\frac{\alpha\beta}{1+\beta}}\] choosing $ |x-y|=|t-s|^{\frac{\alpha}{1+\beta}} $.
	\end{proof}
\end{prop}

We are now ready to prove the following Theorem.
\begin{theorem}\label{thm.loc.existence.reg.}
Let $ u_0\in C^{0,1}\left( \R\right) $ be bounded with $u_0(0)=0 $, $ u_0(y)>0 $ if $ y<0 $ and $ u_0(y)<0 $ if $ y>0 $. Moreover, let $ \delta\in\left(0,\frac{1}{2}\right) $ and $ u_0\left.\right|_{\R_\pm}\in C^{2,\delta}(\R_\pm) $. Then, for a time $ t^*>0 $ small enough there exists a unique solution 
$$ (u_1,u_2,s)\in \mathcal{C}^{1,2}_{t,y}\left((0,t^*)\times\R_- \right)\times  \mathcal{C}^{1,2}_{t,y}\left((0,t^*)\times\R_+ \right)\times C^1([0,t^*]) $$ to the problem \eqref{syst.6}. Moreover, $$(u_1,u_2,\dot s)\in \mathcal{C}^{\delta/2,1+\delta}_{t,y}\left((0,t^*)\times\R_- \right)\times  \mathcal{C}^{\delta/2,1+\delta}_{t,y}\left((0,t^*)\times\R_+ \right)\times C^{\delta/2}([0,t^*])  $$ for $ \delta<\frac{1}{2} $. Furthermore, for any $ \eps>0 $ it is also true that $ u_1\in \mathcal{C}^{1+\delta/2,2+\delta}_{t,y}([\eps,t^*]\times\R_-)  $ as well as $ u_2\in \mathcal{C}^{1+\delta/2,2+\delta}_{t,y}([\eps,t^*]\times\R_+)$.
\begin{proof}
We have to show that the fixed-point solution $ (u_1,u_2,s) $ found in Theorem \ref{thm.local.well.pos.} has the desired regularity.  We already know that the interface $ s(t)=\int_0^t \dot s(\tau)d\tau $ solves in a classical sense the equation \[	\dot{s}(t)=\frac{1}{L}\left( \partial_{y} u_2(t,0^+)-K\partial_{y} u_1(t,0^-)\right)\] with initial value $ s(0)=0 $. Moreover, $ s\in C^1([0,t^*]) $. 

We will now show that $ (u_1,u_2)\in  \mathcal{C}^{1,2}_{t,y}\left((0,t^*)\times\R_- \right)\times  \mathcal{C}^{1,2}_{t,y}\left((0,t^*)\times\R_+ \right)  $. This will imply that they solve the parabolic equations \eqref{parabolic.eq} strongly in $ (0,t^*]\times \R_\pm $. For the fixed-point solution $( u_1,u_2,\dot s) $ of Theorem \ref{thm.local.well.pos.} we define the sources
\[F_1(t,y)=\dot s(t)\partial_y u_1(t,y)\text{ for } y<0\] and \[ F_2(t,y)=\dot s(t)\partial_y u_2(t,y)+I_\alpha[u_2+T_M] \text{for } y>0.\] We will show that $ F_i\in C^{\delta/2,\delta}([0,t^*],\R_\pm) $, this will imply the regularity of the functions $ u_1,u_2 $.


We first show that $ u_1,\;u_2 $ are $ \frac{\alpha}{2} $-H\"older in time. To this end we define \[ \bar{u}_{0,1}(y)=\begin{cases}
	u_0(y)&y\leq0\\-u_0(-y)&y>0
\end{cases}\quad\text{ and }\quad\bar{u}_{0,2}(y)=\begin{cases}
	u_0(y)&y\geq0\\-u_0(-y)&y<0
\end{cases} .\] 
Since the continuous function $ u_0$ satisfies $u_0\in C^{2,\delta}(\R_\pm) $, then $ \bar{u}_{0,i}\in C^{1,1}(\R_\pm) $. Hence, Lemma \ref{bnd.value.estimates} implies that $  \bar{u}_{0,i}*_y\Phi\in \mathcal{C}^{1/2,1+1}_{t,y}([0,t^*]\times \R_\pm) $ as well as $ \bar{u}_{0,i}*_y\Phi\in \mathcal{C}^{1+1/2,2+1}_{t,y}([\eps,t^*]\times \R_\pm)  $ for any $ \eps>0 $.
We also know that $ F_i\in \mathcal{C}^{0,0}_{t,y}\in\left([0,t^*],\R_\pm\right) $. Therefore Lemma \ref{interior.estimate} implies that $ F_i*G\in \mathcal{C}^{\alpha/2,1+\beta}_{t,y}\left([0,t^*],\R_+\right) $ for any $ \alpha,\beta\in(0,1) $. Thus, \[u_i=u_{0,i}*_y\Phi+F_i*G\in \mathcal{C}^{\alpha/2,1+\beta}_{t,y}\left([0,t^*],\R_\pm\right) \text{ for any }\alpha,\beta\in(0,1).\]
This result has two important consequences. First of all, Proposition \ref{hoelder} implies that the functions $ \partial_yu_i $ are $ \frac{\delta}{2} $-H\"older in time for some $ \delta\in\left(0,\frac{1}{2}\right) $. Indeed, it is not difficult to see that for any $ 0<\delta<\frac{1}{2} $ there exists $ \alpha,\beta \in(0,1)$ such that $ \frac{\alpha\beta}{1+\beta}=\delta $. This implies also that the derivatives $ \partial_yu_i(t,0) $ are $ \frac{\delta}{2} $-H\"older in time and thus by definition $ \dot{s}(t)\in C^{\delta/2}([0,t^*]) $. This yields further a better regularity for $ F_1 $, indeed $ F_1\in \mathcal{C}^{\delta/2,1+\delta}\left([0,t^*],\R_-\right) $. 

A similar result can be shown also for $ F_2 $. We first of all remark that since $ u_2 $ is bounded, then $ (u_2+T_M)\in \mathcal{C}^{\alpha/2,1+\beta}_{t,y}\left([0,t^*],\R_+\right) $ for any $ \alpha,\beta \in (0,1) $. Thus, we have to show that $ I_\alpha[u_2+T_M]\in \mathcal{C}^{\delta/2,\delta}([0,t^*]\times \R_+) $. Clearly, it is $ \delta/2 $-H\"older in time, since $ u_2 $ is so. For the space variable we use that $ E_1\in L^1(\R)\cap L^2(\R) $, hence by interpolation also $ E_1\in L^q(\R) $ for any $ q\in [1,2] $. Let $ b>a>0 $, then for any $ \delta\in[0,1/2] $ we have
\[\int_a^b \frac{\alpha}{2} E_1(\alpha \eta)d\eta \leq \frac{1}{2} |a-b|^{\delta} \Arrowvert E_1\Arrowvert_{L^{\frac{1}{1-\delta}}}.\]
Therefore, for $ v\in C^{0,\delta}(\R_+) $ and $ y>x>0 $ we estimate
\begin{multline}\label{hoeld.of.I}
\left|\int_0^\infty \frac{\alpha}{2}v^4(\eta)\left(E_1(\alpha(y-\eta))-E_1(\alpha(x-\eta))\right)d\eta\right|=\left|\int_{-y}^\infty \frac{\alpha}{2}v^4(\eta+y)E_1(\alpha\eta)d\eta-\int_{-x}^\infty \frac{\alpha}{2}v^4(\eta+x)E_1(\alpha\eta)d\eta\right|\\\leq\left| \int_{-x}^\infty \frac{\alpha}{2}E_1(\alpha\eta)\left(v^4(\eta+y)-v^4(\eta+x)\right)d\eta\right|+\left|\int_{-y}^{-x}\frac{\alpha}{2}E_1(\alpha\eta) v^4(\eta+y)d\eta\right|\\
\leq \Arrowvert v^4\Arrowvert_{\delta}|x-y|^\delta+\frac{1}{2}\Arrowvert v^4\Arrowvert_{C^0}\Arrowvert E_1\Arrowvert_{L^{\frac{1}{1-\delta}}}|x-y|^{\delta}.
\end{multline} 
Hence, we can conclude that $ I_\alpha[u_2+T_M] \in \mathcal{C}^{\delta/2,\delta}([0,t^*]\times\R_+)$ and consequently that $ F_2\in \mathcal{C}^{\delta/2,\delta}([0,t^*]\times\R_+)$.\\

In order to prove finally that $ (u_1,u_2) $ is a classical solution to \eqref{parabolic.eq} we use that any bounded solution $ w_i\in \mathcal{C}^{1,2}_{t,y}([0,t^*]\times \R_\pm) $ of the heat equation 
\begin{equation}\label{heat.eq}
	\begin{cases}
		\partial_t w_i(t,y)-\partial_y^2w_i(t,y)=F_i(t,y)& (0,t^*)\times \R_\pm\\
		w_i(0,y)=u_{0}(y)& y\in\R_\pm\\
		w_i(t,0)=0 & t\in [0,t^*]
	\end{cases}
\end{equation}
 can be written both by 
 \[w_i=\bar{u}_{0,i}*_y\Phi+F_i*G\]
 and by the sum $ w_i=v_i+h_i $ of two functions $ v_i,h_i $ solutions to an inhomogeneous heat equation in the whole space and a homogeneous equation in the half-space, respectively. 
 
 Let us consider the even extensions of $ F_i $
 \[	\bar{F}_1(t,y)=\begin{cases}
 F_1(t,y)&y\leq 0\\F_1(t,-y)&y>0
 \end{cases} \text{ and }	\bar{F}_2(t,y)=\begin{cases}
 F_2(t,y)&y\geq 0\\F_2(t,-y)&y<0.
\end{cases}\]
Then, $ \bar{F}_i\in \mathcal{C}^{\delta/2,\delta}([0,t^*],\R) $. By Proposition \ref{reg.heat} and Lemma \ref{bnd.value.estimates} we see that 
\[ v_i(t,y):=\left(\bar{u}_{0,i}*_y\Phi\right)(t,y)+\left(\bar{F}_i*\Phi\right)(t,y)\in \mathcal{C}_{t,y}^{1+\delta/2,2+\delta}([\eps,t^*]\times \R)\] for any $ \eps>0 $. Thus, $ v_i\in \mathcal{C}_{t,y}^{1,2}((0,t^*]\times \R) $ is a strong solution to 
\begin{equation*}
	\begin{cases}
		\partial_t v_i(t,y)-\partial_y^2v_i(t,y)=\bar{F}_i(t,y)& (0,t^*)\times \R\\
		v_i(0,y)=\bar{u}_{0,i}(y)& y\in\R\\
	\end{cases}
\end{equation*}
Moreover, since $ v_i(t,0)=\left(\bar{F}_i*\Phi\right)(t,0) $ by Proposition \ref{reg.heat} we see that $ v_i(t,0)\in \mathcal{C}_{t,y}^{1,\delta/2}([0,t^*]) $. Thus, by Proposition \ref{reg.heat} we can conclude that
\[ h_i(t,y):=2\left(\partial_y \Phi *_t v(\cdot, 0)\right)(t,y)\in \mathcal{C}_{t,y}^{1+\delta/2,2+\delta}([\eps,t^*]\times \R_\pm)\]
for all $ \eps>0 $. Using the property of the double-layer potential (c.f. \cite{Ladyzenskaja}) we obtain that $ h_i\in \mathcal{C}_{t,y}^{1,2}((0,t^*),\R_\pm) $ is a strong solution to
\begin{equation*}
	\begin{cases}
		\partial_t h_i(t,y)-\partial_y^2h_i(t,y)=0& (0,t^*)\times \R_\pm\\
		h_i(0,y)=0& y\in\R_\pm\\
		h_i(t,0)=-v_i(t,0)&t\in[0,t^*].
	\end{cases}
\end{equation*}
Thus, it is not difficult to see that $ w_i=v_i+h_i $ is a strong solution of \eqref{heat.eq} in the interior $ (0,t^*)\times\R_\pm $. Moreover, $ w_i $ is the unique bounded solution of \eqref{heat.eq} and it has the same integral representation of $ u_i $. Hence $ u_i=w_i $ and using the regularity of $ v_i,h_i $ we conclude
\[u_i\in \mathcal{C}_{t,y}^{1,2}((0,t^*),\R_\pm) \text{ and }u_i\in \mathcal{C}_{t,y}^{1+\delta/2,2+\delta}([\eps,t^*]\times \R_\pm)\text{ for all }\eps>0.\]
\end{proof}
\end{theorem}
\subsection{Maximum principle}
The local solutions $ u_1,u_2 $ were obtained as $ u_i=T_i-T_M $, where $ T_M$ is the melting temperature and $ T_1,T_2 $ are the solutions to \eqref{syst.5}. Since the temperature is a non-negative quantity, we want to show that the solutions $ T_i $ are non-negative as long as they are defined. Moreover, we expect that in the liquid, i.e. $ y<0 $, the temperature satisfies $ T_1>T_M $ while in the solid ($ y>0 $) the temperature satisfies $ T_2<T_M $. In bounded domains the maximum principle yields these results. Since we are working on an unbounded domain we will first consider some suitable problems in bounded domains, where the maximum principle assures the desired properties of the temperature, and then we will show that their solutions converge to $ u_1 $ and $ u_2 $, the solutions to the problem in the whole space.

For $ R\geq 2 $ we fix $ \eta_R\in C^\infty(\R) $ with the property that $ \eta_R\equiv 1 $ if $ |y|\leq R-1 $ and $ \eta_R\equiv 0 $ if $ |y|\geq R $ and $ |\eta|\leq 1 $. Moreover, we choose $ \eta_R $ with $ \Arrowvert \eta'_R\Arrowvert_\infty \leq 2 $ as well as $ \max\left\{\Arrowvert \eta''_R\Arrowvert_\infty,\Arrowvert\eta'''_R\Arrowvert_\infty\right\}\leq C $ for a fixed constant $ C>0 $ independent of $ R $. We will consider $ u_1^R $ and $ u_2^R $ solutions to \begin{align}\label{cut-off}
	\begin{cases}
		\partial_t u_1^R-\dot{s}\partial_y u_1^R=\kappa\partial_y^2 u_1^R &  (0,t_*)\times (-R,0)\\
		u_1^R(t,0)=0\\
		u_1^R(t,-R)=0\\
		u_1^R(0,y)=u_0(y)\eta_R(y)
	\end{cases}
	& \text{and } \begin{cases}
		\partial_t u_2^R-\dot{s}\partial_y u_2^R=\partial_y^2 u_2^R-I_\alpha^R[u_2^R+T_M] &  (0,t_*)\times (0,R)\\
		u_2^R(t,0)=0\\
		u_2^R(t,R)=0\\
		u_2^R(0,y)=u_0(y)\eta_R(y)
	\end{cases}
\end{align}
$ I_\alpha^R[v](t,y)=v^4(t,y)-\int_0^R \frac{\alpha}{2}E_1(\alpha(y-\eta))v^4(t,\eta)d\eta $ and $ \dot{s} $ is the time derivative of the moving interface $ s $, which is together with $ u_1,u_2 $ of Theorem \ref{thm.loc.existence.reg.} the unique solution of problem \eqref{syst.6} for $ t\in(0,t_*) $. We will show the following Lemma.
\begin{lemma}\label{lemma.cut.off}
	Let $ R\geq 2 $. Let $ u_0 $ be as in Theorem \ref{thm.loc.existence.reg.} and let $ \eta_R $ as above. Then there exist a time $ 0<t_*\leq t^* $ small enough and independent of $ R $ such that there exist unique solutions $ u_1^R\in \mathcal{C}^{1,2}_{t,y}((0,t_*)\times[-R,0]) $ and $ u_2^R\in \mathcal{C}^{1,2}_{t,y}((0,t_*)\times[0,R]) $ to \eqref{cut-off} which satisfy the H\"older regularity \[u_1^R\in \mathcal{C}^{\delta/2,1+\delta}_{t,y}([0,t_*]\times[-R,0]) \text{ and }u_1^R\in \mathcal{C}^{\delta/2,1+\delta}_{t,y}([0,t_*]\times[0,R])\] with uniformly bounded H\"older norms.\\
	
	Let also $ u_i $ be the solutions to \eqref{syst.6} of Theorem \ref{thm.loc.existence.reg.}, then there exist two subsequences $ u_i^{R_n} $ which converge to $ u_i $ as $ n\to\infty $ uniformly in every compact set and pointwise everywhere.
	\begin{proof}
		We consider the Green's function for the heat equation on the interval $ [-R,0] $ and $ [0,R] $ given by \[ \tilde{G}_R(y,\xi,at)=\sum_{n\in\mathbb{Z}}\Phi(y-\xi-2nR,at)-\Phi(y+\xi-2nR,at).\] Let $ F_1^R(t,y)=\dot{s}(t)\partial_y u_1^R(y) $ and $ F_2^R(t,y)=\dot{s}(t)\partial_y u_1^R(t,y)-I_\alpha^R[U_2^R+T_M](t,y) $. We obtain the following fixed-point representations for the solutions to \eqref{cut-off} 
		\begin{equation}\label{cut-off.1}
			u_1^R(t,y)=\mathcal{L}_1^R(u_1^R)(t,y)=\int_{-R}^0 u_0(\xi)\eta_R(\xi)\tilde{G}_R(y,\xi,\kappa t)d\xi+ \int_{-R}^0 \int_0^t F_1^R(\tau,\xi)\tilde{G}_R(y,\xi,\kappa(t-\tau))d\tau d\xi,
		\end{equation}
		and
		\begin{equation}\label{cut-off.2}
			u_2^R(t,y)=\mathcal{L}_2^R(u_2^R)(t,y)=\int_0^R u_0(\xi)\eta_R(\xi)\tilde{G}_R(y,\xi,t)d\xi+ \int_0^R\int_0^t F_2^R(\tau,\xi)\tilde{G}_R(y,\xi,(t-\tau))d\tau d\xi.
		\end{equation}
		We will prove the existence of a unique fixed-point for the operators $ \mathcal{L}_1^R: \mathcal{A}_1^R\to  \mathcal{C}_{t,y}^{0,1}\left([0,t_*]\times [0,R]\right) $ and $\mathcal{L}_2^R: \mathcal{A}_2^R\to\mathcal{C}_{t,y}^{0,1}\left([0,t_*]\times [-R,0]\right) $ for all $ R\geq2 $, where
		\[\mathcal{A}_1^R=\left\{u\in \mathcal{C}_{t,y}^{0,1}\left([0,t_*]\times [-R,0]\right):\;\Arrowvert u\Arrowvert_{0,1}\leq C_1\right\}\quad\text{and}\quad \mathcal{A}_2^R=\left\{u\in \mathcal{C}_{t,y}^{0,1}\left([0,t_*]\times [0,R]\right):\;\Arrowvert u\Arrowvert_{0,1}\leq C_2\right\}\]
		for constants $ (1-\theta)C_i>4\Arrowvert u_0\Arrowvert_{1}$ for a fixed $ \theta\in(0,1) $.\\
		
		First of all we see that it is possible to extend in an odd manner $ u_0\eta_R $ to the whole real line. Indeed, we can consider as usual the odd extension of the initial value as
		\[\tilde{u}_{0,1}^R(y)=\begin{cases}
			u_0(y)\eta_R(y)&-R\leq y\leq 0\\ -u_0(-y)\eta_R(-y)&0<y\leq R
		\end{cases} \quad\text{ and }\quad\tilde{u}_{0,2}^R(y)=\begin{cases}
			u_0(y)\eta_R(y)&0\leq y\leq R\\ -u_0(-y)\eta_R(-y)&-R\leq y<0.
		\end{cases}\]
		We define
		\[\bar{u}_{0,1}^R(y)=\bar{u}_{0,1}^R(y+2nR)
		\text{ and }\tilde{u}_{0,2}^R(y)=\bar{u}_{0,2}^R(y+2nR)\text{ for }y\in[-(2n+1)R,(-2n+1)R]\]
		One can easily see then that
		\begin{equation*}\label{initial.value}
			\left(\bar{u}_{0,1}^R*\Phi(\cdot, \kappa\cdot)\right)(t,y)=\int_{-R}^0 u_0(\xi)\eta_R(\xi)\tilde{G}_R(y,\xi,\kappa t)d\xi \text{ and }\left(\bar{u}_{0,2}^R*\Phi(\cdot,\cdot)\right)(t,y)=\int_{0}^R u_0(\xi)\eta_R(\xi)\tilde{G}_R(y,\xi, t)d\xi.
		\end{equation*}
		Similarly, it is not difficult to see that with a change of coordinate
		\begin{equation}\label{green.fct}
			\int_0^{\pm R} \left|\tilde{G}_R(y,\xi,at)\right|d\xi\leq \sum_{n\in\mathbb{Z}}\int_{2nR}^{(2n+1)R}\left|\Phi(y-\xi,at)\right|+\left|\Phi(y+\xi,at)\right|d\xi\leq 2\int_\R|\Phi(\xi,at)|d\xi\leq 2.
		\end{equation}
		The same estimate holds for the integral in $ [-R,0] $. In the same way we have also 
		\begin{equation}\label{green.fct.der}
			\int_0^{\pm R} \left|\partial_y\tilde{G}_R(y,\xi,at)\right|d\xi\leq \sum_{n\in\mathbb{Z}}\int_{2nR}^{(2n+1)R}\left|\partial_\xi\Phi(y-\xi,at)\right|+\left|\partial_\xi\Phi(y+\xi,at)\right|d\xi\leq 2\int_\R|\partial_\xi\Phi(\xi,at)|d\xi\leq \frac{2\sqrt{2}}{\sqrt{at}}.
		\end{equation}
		Thus, for $ u_i,v_i\in\mathcal{A}_i^R $ we have
		\begin{equation}\label{R1}
		\begin{split}
				\Arrowvert \mathcal{L}_1^R(u_1)\Arrowvert_{0,1}\leq& 4\Arrowvert u_0\Arrowvert_1+2\Arrowvert \dot s\Arrowvert_\infty C_1\left(t_*+\frac{\sqrt{2t_*}}{\sqrt{\kappa}}\right) \quad\text{ and }\\	\Arrowvert \mathcal{L}_1^R(u_2)\Arrowvert_{0,1}\leq& 4\Arrowvert u_0\Arrowvert_1+2\left(\Arrowvert \dot s\Arrowvert_\infty C_2+2\left(C_2+T_M\right)^4\right)\left(t_*+\sqrt{2t_*}\right) 
		\end{split}
		\end{equation}
		and 
		\begin{equation}\label{R2}
			\Arrowvert \mathcal{L}_1^R(u_1)-\mathcal{L}_1^R(v_1)\Arrowvert_{0,1}\leq 2\Arrowvert \dot s\Arrowvert_\infty\Arrowvert u_1-v_1\Arrowvert_{0,1}\left(t_*+\frac{\sqrt{2\bar t}}{\sqrt{\kappa}}\right) 
		\end{equation}
		as well as
		\begin{equation}\label{R3}
			\Arrowvert \mathcal{L}_1^R(u_2)-\mathcal{L}_1^R(v_2)\Arrowvert_{0,1}\leq 2\Arrowvert u_2-v_2\Arrowvert_{0,1}\left(\Arrowvert \dot s\Arrowvert_\infty +16\left(C_2^3+T_M^3\right)\right)\left(t_*+\sqrt{2\bar t}\right). 
		\end{equation}
	Since the estimates \eqref{R1}, \eqref{R2} and \eqref{R3} do not depend on $ R $, there exists $ 0<t_*\leq t^* $ independent of $ R $ and small enough such that by the Banach fixed-point Theorem there exist unique fixed-points $ u_1^R\in \mathcal{C}^{0,1}_{t,y}([0,t_*]\times [-R,0]) $ and $ u_2^R\in \mathcal{C}^{0,1}_{t,y}([0,t_*]\times [0,R]) $ of the operators $ \mathcal{L}_1^R $ and $ \mathcal{L}_2^R $ of \eqref{cut-off.1} and \eqref{cut-off.2}, respectively. 
		
		Adapting Lemma \ref{interior.estimate} to the Green's function $ \tilde{G}_R $ one can prove that for any function $ f_\pm\in \mathcal{C}^{0,0}_{t,y}([0,t_*],[0,\pm R]) $ the convolution $ f_\pm *\tilde{G}_R\in \mathcal{C}_{t,y}^{\alpha/2,1+\beta}([0,t_*],[0,\pm R]) $ for any $ \alpha,\beta\in(0,1) $. We omit the proof since it an easy calculation based on the proof of Lemma \ref{interior.estimate} and on a suitable change of variables as in estimates \eqref{green.fct} and \eqref{green.fct.der}.
		This result together with Lemma \ref{bnd.value.estimates} applied to the odd extensions $ \bar{u}_{0,i}^R $ for $ i=1,2 $ implies the desired  H\"older regularity
		\[u_1^R\in \mathcal{C}_{t,y}^{\delta/2,1+\delta}([0,t_*]\times[-R,0]) \text{ and }u_2^R\in \mathcal{C}_{t,y}^{\delta/2,1+\delta}([0,t_*]\times[0,R])\] for any $ \delta\in(0,1) $. Moreover, the H\"older norms are uniformly bounded in $ R $. Indeed, $ \Arrowvert F_1^R\Arrowvert_\infty\leq C_1 \Arrowvert \dot{s}\Arrowvert_\infty  $ and $ \Arrowvert F_2^R\Arrowvert_\infty\leq C_2\Arrowvert \dot{s}\Arrowvert_\infty +2(C_2+T_M)^4 $ and hence Lemma \ref{bnd.value.estimates} and Lemma \ref{interior.estimate} yield
		\[\Arrowvert u_i^R\Arrowvert_{\delta/2,1+\delta}\leq \max\left\{C_i, \Arrowvert u_0\Arrowvert_{1}, C\left(t_*, t_*^{(1-\delta)/2}\right) \frac{\Arrowvert F_i^R\Arrowvert_\infty}{1-\delta}\right\} \quad \text{ for }i=1,2, \]
		where $ C\left(t_*, t_*^{(1-\delta)/2}\right)>0 $ does not depend on $ R $.\\
		
		Before proving that the functions $ u_i^R $ are also classical solutions to \eqref{cut-off}, we prove the convergence result. Let us extend $ u_i^R $ for $ |y|>R $ by a function $ \bar{u}_i^R\in \mathcal{C}_{t,y}^{\delta/2,1+\delta}([0,t_*],\R_\pm) $ with norm $ \Arrowvert \bar{u}_i^R\Arrowvert_{\delta/2,1+\delta}\leq 2\Arrowvert u_i^R\Arrowvert_{\delta/2,1+\delta} $. Then the uniform boundedness of the H\"older norm implies that $ \bar{u}_i^R $ is a compact sequence in $ \mathcal{C}_{t,y}^{\alpha/2,1+\alpha}([0,t_*],[a,b]) $ for every compact set $ [a,b]\in\R_\pm $ and for $ \alpha<\delta $. Therefore, for any sequence $ R_n\to\infty $ a diagonal argument yields the existence of a subsequence $ R_{n_k} $, which we will denote for simplicity by $ R_n $, and of a function $ \bar{u}_i\in \mathcal{C}_{t,y}^{0,1}([0,t_*]\times\R_\pm)\cap C_{\text{loc}}^{\alpha/2,1+\alpha}([0,t_*]\times\R_\pm)  $ such that \[\bar{u}_i^{R_n}\to \bar{u}_i \text{ and }\partial_y\bar{u}_i^{R_n}\to \partial_y\bar{u}_i\text{ as }n\to\infty\]
		uniformly in every compact set and pointwise everywhere. 
		
		We show now $ \bar{u}_i=u_i $. It is not difficult to see that $ u_0\eta_{R_n}\to u_0 $ pointwise for $ y\in\R $. Notice that also $ u_i^{R_n}(y)\to \bar{u_i}(y) $  and $ \partial_yu_i^{R_n}(y)\to \partial_y\bar{u_i}(y) $ for $ R_n>|y| $ and $ n\to\infty $, and that
		\[\left|\int_0^{R_n} \frac{\alpha}{2}E_1(\alpha(\eta-\xi))\left(u_2^{R_n}+T_M\right)^4d\eta-\int_0^\infty \frac{\alpha}{2}E_1(\alpha(\eta-\xi))\left(\bar{u}_2+T_M\right)^4d\eta\right|\to 0 \text{ as }n\to\infty.\]
		Hence, we see that\[F_1^{R_n}(t,y)\to \bar{F}_1(t,y)=\dot s(t)\partial_y \bar{u}_1(t,y)\text{ for } y<0 \text{ and } n\to\infty\] and \[F_2^{R_n}(t,y)\to \bar{F}_2(t,y)=\dot s(t)\partial_y \bar{u}_2(t,y)+I_\alpha[\bar{u}_2+T_M] \text{for } y>0\text{ and } n\to\infty.\]
		Using that \[ \tilde{G}_{R_n}(y,\xi,at)=G(y,\xi,at)+\sum_{|n|\geq1}G(y-2nR,\xi,at)\to G(y,\xi,at) \text{ as }n\to\infty\]
		we can conclude using Lebesgue dominated convergence theorem that the functions $ \bar{u}_i $ solve the integral equations
	\begin{equation}\label{ubar}
		\bar{u}_i(t,y)=\int_{\R_\pm} u_0(\xi)G(y,\xi,a_i t)d\xi+ \int_0^t\int_{\R_\pm}  \bar{F}_i(\tau,\xi)G(y,\xi,a_i(t-\tau))d\tau d\xi,
	\end{equation}
		where $ a_1=\kappa $ and $ a_2=0 $. An easy application of Banach fixed-point theorem shows that \eqref{ubar} has a unique fixed-point for bounded functions in $ \mathcal{C}_{t,y}^{0,1}\left([0,t_*],\R_{\pm}\right) $ with bounded derivative. Thus, since $ u_i $ solves also \eqref{ubar}, we can conclude that $ \bar{u}_i=u_i $.\\

		Finally, we prove that the functions $ u_i^R $ are classical solutions to the heat equations \eqref{cut-off}. Clearly, by Lemma \ref{bnd.value.estimates} we have that $ \bar{u}_{0,i}^R*\Phi\in \mathcal{C}_{t,y}^{1+\delta/2,2+\delta}([\eps,t_*]\times[0,\pm R]) $. We need to prove the differentiability of $ F_i^R*\tilde{G}_R $. First of all we notice that because of $ \partial_t \tilde{G}_R(y,\xi,at)=a\partial_y^2\tilde{G}_R(y,\xi,at) $ we only need to show that there exists the second spatial derivative since also
		\[\lim\limits_{\eps\to 0}\int_0^{\pm R} F_i^R(t-\eps,\xi)\tilde{G}_R(y,\xi,a\eps)d\xi=\lim\limits_{\eps\to 0}\int_0^{\pm R} F_i^R(t-\eps,\xi)\Phi(y-\xi,a\eps)d\xi=F_i^R(t,y)\]
		for any $ y\in(-R,0) $ or $ y\in(0,R) $. Thus, we compute using the change of coordinates as in \eqref{green.fct} and \eqref{green.fct.der}
		\begin{equation*}
			\begin{split}
				\left|\int_0^t \right.&\left.\int_0^{\pm R} F_i^R(t-\tau,\xi)\partial_y^2\tilde{G}_R(y,\xi,a\tau)d\xi d\tau\right|\\\leq&
				\left|\int_0^t \int_0^{\pm R} \left(F_i^R(t-\tau,\xi)-F_i^R(t-\tau,y)\right)\partial_y^2\tilde{G}_R(y,\xi,a\tau)d\xi d\tau\right|+	\left|\int_0^tF_i^R(t-\tau,y) \int_0^{\pm R} \partial_y^2\tilde{G}_R(y,\xi,a\tau)d\xi d\tau\right|\\
				\leq& 2\Arrowvert F_i^R\Arrowvert_\delta\int_0^t\int_\R |y-\xi|^\delta \left|\partial^2_\xi\Phi(\xi,a\tau)\right|+\left|\int_0^tF_i^R(t-\tau,y) \int_0^{\pm R} \partial_\xi^2\tilde{G}_R(y,\xi,a\tau)d\xi d\tau\right|<\infty.
			\end{split}
		\end{equation*}
		We remark that we obtain the first term since $ |y-(\eta+2nR)|\leq \min\{|y-\eta|,|y+\eta|\} $ for any $ \eta\in[-2nR,(-2n+1)R] $. Moreover, the boundedness of the first term is a well-known property of the fundamental solution of the heat equation. The boundedness of the second term is a classical result in parabolic theory which combines the fact that 
		\[\begin{split}
			\left|\int_0^t \right.&\left.F_i^R(t-\tau,y)\int_0^{\pm R} \partial_\xi^2\Phi(y\pm\xi-2nR,a\tau)d\xi d\tau\right|\\&=\left|\int_0^t F_i^R(t-\tau,y)\left( \partial_\xi\Phi(y\pm R-2nR,a\tau)-\partial_\xi\Phi(y-2nR,a\tau)\right)d\xi d\tau\right|\leq C\Arrowvert F_i^R\Arrowvert_\infty
		\end{split}\]
		and the fact that the tail of the series $ \int_0^t \int_0^{\pm R} \left|\partial_\xi^2\tilde{G}_R(y,\xi,a\tau)\right|d\xi d\tau $ converges. Finally, the H\"older continuity of $ F_i^R $ is due to the H\"older regularity of $ u_i^R $ and of $ \dot{s} $ as well as to the convolution with the exponential integral as in \eqref{hoeld.of.I}. Thus, we conclude that $ u_1^R\in \mathcal{C}_{t,y}^{1,2}((0,t_*)\times[-R,0]) $ and $ u_1^R\in \mathcal{C}_{t,y}^{1,2}((0,t_*)\times[0,R]) $ are classical solutions of \eqref{cut-off}.
	\end{proof}
\end{lemma}

This approximation result will be used in order to show that $ u_1>0 $ as well as $ 0<u_2<-T_M $. Before applying the maximum principle though, we need to show that the maximal interval of existence of the solutions for the original equation \eqref{syst.6} can be approximated by the one of the solutions to \eqref{cut-off}. This is due to the uniform convergence in compact domain of any subsequence of solutions to \eqref{cut-off}. Thus, the norms of the convergent sequence are uniformly bounded in time.
\begin{lemma}\label{lemma.interval.existence}
Let $ [0,t^*] $ be the maximal interval of existence for the solution $ (u_1,u_2,\dot{s}) $ to the problem \eqref{syst.6}. For any $ \eps>0 $ there exists a sequence $ \left(u_1^{R^\eps_n},u_2^{R^\eps_n}\right) $ solving \eqref{cut-off} on $ [0,t^*-\eps] $ with
\[u_i^{R^\eps_n}\to u_i\text{ as }n\to\infty\] uniformly in every compact set and for $ i=1,2 $.
\begin{proof}We argue by contradiction. For any sequence $ \left\{R_n\right\}_{n\in\mathbb{N}}$ with $R_n\to \infty$ as $ n\to\infty $ we define the maximal time of existence of the sequence $ u_i^{R_n} $ by
\[t_*(R_n):=\sup\left\{t_*>0: u_i^{R_n} \text{ exists in }[0,t_*]\text{, for all }n\text{ and }i=1,2\right\}.\] By the convergence result of Lemma \ref{lemma.cut.off} we know that $ t_*(R_n)\leq t^* $. Hence we consider
\[\bar t:=\sup\{t_*(R_n): R_n \text{ is an increasing diverging sequence}\}\leq t^*.\] If $ \bar{t}=t^* $ then Lemma \ref{lemma.interval.existence} is proved. Indeed, by the definition of $ \bar{t} $ for any $ \eps>0 $ there exists an increasing diverging sequence $ R_n^\eps $ such that $ t_*(R_n^\eps)>\bar{t}-\eps=t^*-\eps $. Thus, taking a suitable subsequence we conclude the lemma.

Let us assume that $ \bar{t}<t^* $. Let also $ \delta>0 $ and let $ R^{\delta}_n $ be an increasing sequence such that $ t_*\left(R^{\delta}_n\right)>\bar t-\delta $. By Lemma \ref{lemma.cut.off} there exists a subsequence $ R^\delta_{n_k} $ such that $ u_i^{R^{\delta}_{n_k}}\to u_i $ uniformly in every compact set. We hence see by the convergence result that the norms of $ u_i^{R^{\delta}_{n_k}} $ are uniformly bounded on $ [0,t_*\left(R^{\delta}_n\right)] $ by the (bounded) norm of $ u_i $ on the larger interval time $ [0,\bar t] $. Thus, the solutions $ u_i^{R^{\delta}_{n_k}} $ can be extended for larger times so that
\[t_*\left(R^{\delta}_{n_k}\right)\geq t_*\left(R^{\delta}_n\right)+\theta(\bar t)>\bar t-\delta+\theta(\bar t),\] where $ \theta(\bar t)>0 $ depends on the norm of $ u_i $ on $ [0,\bar t] $ and not on $ \delta $. Since $ \delta $ is arbitrary we obtain the contradiction $ \bar t\geq \bar{t}+\theta(\bar{t})$. This concludes the proof of this lemma.\end{proof}
\end{lemma}
We can now prove with the maximum principle the following proposition.
\begin{prop}[Properties of the solution]\label{prop.sol.}
	Let $ u_0\in C^0_b(\R) $ be as in the assumptions of Theorem \ref{thm.loc.existence.reg.}. Let $ (u_1,u_2,\dot{s}) $ be the local solution to \eqref{syst.6} of Theorem \ref{thm.loc.existence.reg.} for $t\in [0,t^*]$, which is the maximal interval of existence. Then $ u_1\geq 0 $ and $ -T_M\leq u_2\leq 0 $. Moreover, $ u_1(t,y)>0 $ for all $ y\in(-a,b)\subset(-\infty,0] $ and $ -T_M<u_2(t,y)<0 $ for all $ y\in(-a,b)\subset[0,\infty) $.
	\begin{proof}
		As we have seen in Lemma \ref{lemma.interval.existence} for any $ \eps>0 $ there exists an increasing diverging sequence $ \left\{R^\eps_n\right\}_n $ such that the solutions $ u_i^{R^\eps_n} $ of \eqref{cut-off} exist on the interval $ [0,t^*-\eps] $ and converge to the solutions $ u_i $ uniformly in every compact set.
		
		First of all, for any $ R>0 $ we apply the classical maximum principle to the functions $ u_1^R $ and $ u_2^R $ solving the parabolic problem \eqref{cut-off} on the bounded domains $ [0,t_*(R)]\times[-R,0] $ and $ [0,t_*(R)]\times[0,R] $. Where for the sake of readability we denote $ t_*=t_*(R) $ the maximal time of existence for the solutions $ u_1^R,\;u_2^R $.\\
		
		Let us first consider $ u_1^R $. Then, since $ u_0(y)\eta_R(y)>0 $ for all $y\in(-R,0) $ and $ u_1^R(t,0)=u_1^R(t-,R)=0 $ for all $ t\in(0,t_*) $, the strong maximum principle for the parabolic equation solved by $ u_1^R $ implies that the minimum is attained only at the parabolic boundary, i.e. $ u_1^R(t,y)>0 $ for all $ (t,y)\in(0,t_*]\times(-R,0) $.
		Let now $ \eps>0 $. Using Lemma \ref{lemma.interval.existence} and the pointwise convergence result of Lemma \ref{lemma.cut.off} we obtain $ u_1(t,y)\geq 0 $ for all $ (t,y)\in[0,t^*-\eps]\times\R_- $. Thus, since $ \eps>0 $ is arbitrary, we conclude $ u_1(t,y)\geq 0 $ in $ [0,t^*]\times\R_- $. Let us now consider $ (a,b)\in\R_- $. On one hand by assumption we know that $ u_0(y)>0 $ for all $y\in(a,b) $, on the other hand we have just seen that $ u_1(t,a),u_1(t,b)\geq 0 $. Applying once more the strong maximum principle to the parabolic equation solved by solution $ u_1 $ on $ (0,t^*)\times(a,b) $ we can conclude that also $ u_1(t,y)>0 $ for all $ y\in(a,b)$.\\
		
		We now pass to the analysis of $ u_2^R $.
		Let us assume that $ u_2^R(t,y)\leq -T_M $ for some $ (t,y)\in[0,t_*]\times[0,R] $. Then, since $ u_0(y)\eta_R(y)>-T_M $, there exists a $ t_0\in(0,t_*] $, the first time such that a $ y_0\in(0,R) $ exist with $ u_2^R(t_0,y_0)=-T_M $. Hence, $ u_2^R(t,y)>-T_M $ for all $ 0\leq t<t_0 $ and $ y\in[0,R] $. This implies $ \partial_t u_2^R(t_0,y_0)\leq 0 $, $ \partial_y u_2^R(t_0,y_0)=0 $ and also $ \partial_y^2 u_2^R(t_0,y_0)\geq 0 $. Moreover, on the one hand $ \left(u_2^R(t_0,y_0)+T_M\right)^4=0 $ and on the other hand there exists an interval $ (0,y_1)\in(0,R) $ such that $ u_2^R(t_0,y)>-T_M $. Hence,
		\begin{multline*}
			0=\partial_t u_2^R(t_0,y_0)-\dot{s}(t_0)\partial_y u_2^R(t_0,y_0)-\partial_y^2 u_2^R(t_0,y_0)+I_\alpha^R[u_2^R+T_M](t_0,y_0)\\\leq -\int_0^R \frac{\alpha}{2}E_1(\alpha(y-\eta))\left(u_2^R(t_0,y)+T_M\right)^4d\eta<0.
		\end{multline*}
		This contradiction implies $ u_2^R(t,y)>-T_M $ for all $ [0,t_*]\times[0,R] $.
		
		Let now $ (t_0,y_0)\in[0,t_*]\times[0,R] $ be such that $ \max\limits_{[0,t_*]\times[0,R]}u_2^R(t,y)=u_2^R(t_0,y_0) $. Assume first that $ (t_0,y_0)\in(0,t_*]\times(0,R) $. Then, $ \partial_t u_2^R(t_0,y_0)\geq0 $, $ \partial_y u_2^R (t_0,y_0)=0 $ and $ \partial_2 u_2^R (t_0,y_0)\leq0 $. Moreover, since $ u_2^R>-T_M $ we have also $ \left(u_2^R(t_0,y_0)+T_M\right)^4\geq \left(u_2^R(t,y)+T_M\right)^4 $. 
		
		This implies
		\begin{multline*}
			0=\partial_t u_2^R(t_0,y_0)-\dot{s}(t_0)\partial_y u_2^R(t_0,y_0)-\partial_y^2 u_2^R(t_0,y_0)+I_\alpha^R[u_2^R+T_M](t_0,y_0)\\\geq\left(\int_{-\infty}^0 \frac{\alpha}{2}E_1(\alpha(y-\eta))d\eta+\int_{R}^\infty \frac{\alpha}{2}E_1(\alpha(y-\eta))d\eta\right)\left(u_2^R(t_0,y_0)+T_M\right)^4\\ +\int_{R}^0 \frac{\alpha}{2}E_1(\alpha(y-\eta))\left[\left(u_2^R(t_0,y_0)+T_M\right)^4-\left(u_2^R(t_0,y)+T_M\right)^4\right]d\eta\\\geq \left(\int_{-\infty}^0 \frac{\alpha}{2}E_1(\alpha(y-\eta))d\eta+\int_{R}^\infty \frac{\alpha}{2}E_1(\alpha(y-\eta))d\eta\right)\left(u_2^R(t_0,y_0)+T_M\right)^4>0.
		\end{multline*}
		This contradiction yields that the maximum is attained at the parabolic boundary, i.e.
		\[u_2^R(t,y)<\max\{u_0(y)\eta_R(y),0\}=0\quad\text{for all}\quad (t,y)\in(0,t_*]\times(0,R).\]
		by the initial condition $ u_2^R(0,y)=\eta_R(y)u_0(y)<0 $ for all $ y\in(0,R) $ we conclude that $ u_2^R(t,y)<0 $ for all $(t,y)\in(0,t_*]\times(0,R)$.
		Let $ \eps>0 $. Using once more Lemma \ref{lemma.interval.existence} and the pointwise convergence result of Lemma \ref{lemma.cut.off} we obtain $ -T_M\leq u_2(t,y)\leq 0 $ for all $(t,y)\in[0,t^*-\eps]\times\R_+$. The erbitrary choice of $ \eps>0 $ implies again that $ -T_M\leq u_2(t,y)\leq 0 $ for all $(t,y)\in[0,t^*]\times\R_+$. To prove that also $ -T_M<u_2(t,y)<0 $ for all $ y\in(a,b)\subset[0,\infty) $ we apply the maximum principle to the function $ u_2 $ again. Let $ R>b $ and assume that $ \min\limits_{[0,t^*]\times[0,R]}u_2=-T_M $. Since $ u_0(y)>-T_M $ for all $ y\in[0,R] $ and $ u_2(t,0)=0$ and $u_2(t,R)\geq -T_M $, there exists $ t_0>0 $ the first time for which there exists some $ y_0\in(0,R] $ such that $ u_2(t_0,y_0)=-T_M $. Hence,
		\begin{multline*}
			0=\partial_t u_2(t_0,y_0)-\dot{s}(t_0)\partial_y u_2(t_0,y_0)-\partial_y^2 u_2(t_0,y_0)+I_\alpha[u_2+T_M](t_0,y_0)\\\leq -\int_0^\infty\frac{\alpha}{2}E_1(\alpha(y-\eta))\left(u_2(t_0,y)+T_M\right)^4d\eta<0,
		\end{multline*}where we used that $ u_2\geq -T_M $ and that the strict inequality holds in a set of positive measure. This contradiction implies that $ u_2(t,y)>-T_M $ for all $ y\in(a,b) $. Let us assume now that $ \max\limits_{[0,t^*]\times[a,b]}u_2=u_2(t_0,y_0)=0 $ for a $ (t_0,y_0)\in(0,t^*]\times(a,b) $. Since $ u_2(t,a),u_2(t,b)\leq 0 $ for $ t> 0 $ we see that
		\begin{multline*}
			0=\partial_t u_2(t_0,y_0)-\dot{s}(t_0)\partial_y u_2(t_0,y_0)-\partial_y^2 u_2(t_0,y_0)+I_\alpha[u_2+T_M](t_0,y_0)\\\geq T_M^4-\int_0^\infty\frac{\alpha}{2}E_1(\alpha(y-\eta))\left(u_2(t_0,y)+T_M\right)^4d\eta\geq T_M^4\left(1-\int_{-b}^\infty\frac{\alpha}{2}E_1(\alpha\eta)d\eta\right)>0,
		\end{multline*} where we also used $ y_0<b $. Thus, since also $ u_0(y)<0 $ for $ y\in(a,b) $ we conclude that $ u_2(t,y)<0 $ for all $ y\in(a,b) $. \\
		
	\end{proof}
\end{prop}

\section{Global well-posedness}\label{subs.global}
In this section we will show that for a class of initial data, the system \eqref{syst.6} has a unique global solution in time. Our aim is to construct a function $ w\in C^{0,1}(\R) $ twice continuously differentiable in $ \R_\pm $ such that $ w(0)=0 $ and such that $ u_1\leq w $ on $ \R_- $ and $ u_2\geq w $ on $ \R_+ $. This would imply global well-posedness. Instead of considering the shifted temperature $ u_i $ we will now study the original temperature $ T_i=u_i+T_M $.

\begin{theorem}[Global Well-posedness]\label{thm.global}
Let $ T_0\in C^{0,1}(\R) $ with $ T_0\left.\right|_{\R_\pm}\in C^{2,\delta}(\R_\pm) $. Let also $ T_0(0)=T_M $, $ T_M<T_0(y)$ for $ y<0 $ and $ 0<T_0(y)<T_M $ for $ y>0 $. Then, if in addition $ \sup\limits_{\R_-}T_0< T_M+\frac{\kappa L^2}{KT_M} $ and $ \inf\limits_{\R_+}T_0>0 $, there is a unique global solution $ (T_1,T_2,s) $ of \eqref{syst.5}.\end{theorem}
Theorem \ref{thm.global} will be proved applying in a suitable way the maximum principle. We will also need the following Lemma, which will be proved at the end of this section.
\begin{lemma}\label{lemma.global}
Let $ T_0 $ be as in the assumption of Theorem \ref{thm.global}. Let $ w\in C^{0,1}(\R) $ be defined by
\begin{equation}\label{W}
w(y)=\begin{cases}
T_M-\frac{\alpha\kappa}{C_1}\left(1-\exp\left(\frac{C_1}{\kappa}y\right)\right)& \text{for}\quad y<0	\\
T_M&\text{ for }y=0\\
 w(y)=T_Me^{-C_2y}\left(1-\frac{T_M^3}{12 C_2^2}+\frac{T_M^3e^{-3C_2y}}{12 C_2^2}\right)&\text{ for }y>0.
\end{cases}
\end{equation}
Then there exist $ C_1,C_2>0 $ satisfying $ C_2>\frac{T_M^{3/2}}{2\sqrt{3}} $, $ C_1>\left(\frac{T_M^5+1}{L^2}\right)^{1/2} $ with $ \Gamma^-(C_1)<C_2<\Gamma^+(C_1) $, where $ \Gamma^\pm(C_1) =\frac{LC_1\pm\sqrt{L^2C_1^2-T_M^5}}{2} $, and $ -\frac{L}{K}C_2<\alpha<0 $ such that $ T_0<w $ on $ \R_- $, $ T_0>w $ on $ \R_+ $ and $ \sup\limits_{\R_\pm}|\partial_y T_0(y)|<|\partial_y w(0^\pm)| $. Moreover, for any $ R>0 $ there exists $ a>0 $ such that $ |T_0(y)-w(y)|>a $ for all $ |y|>R $.
\end{lemma}
We prove now Theorem \ref{thm.global} with the help of Lemma \ref{lemma.global}.
\begin{proof}[Proof of Theorem \ref{thm.global}]
Let $ C_1,C_2>0 $ satisfying $ C_2>\frac{T_M^{3/2}}{2\sqrt{3}} $, $ C_1>\left(\frac{T_M^5+1}{L^2}\right)^{1/2} $ such that $ \Gamma^-(C_1)<C_2<\Gamma^+(C_1) $, where $ \Gamma^\pm(C_1) =\frac{LC_1\pm\sqrt{L^2C_1^2-T_M^5}}{2} $.
Let us also consider a solution $ w $ to
\begin{equation*}\label{w}
	\begin{cases}
		\kappa\partial_y^2 w-C_1\partial_y w=0&y<0\\
		\partial_y^2 w+C_2\partial_y w\geq w^4&y>0\\
		w(0)=T_M\\
		\partial_yw(0^-)>-\frac{L}{K}C_2\\
		\partial_yw(0^+)>-LC_1\\
		w\geq 0
	\end{cases}
\end{equation*}
A simple ODE argument, solving the first equation for $ v=w' $ and integrating, shows that on the negative real line $ w $ is given by
\begin{equation*}\label{w-}
w(y)=T_M-\frac{\alpha\kappa}{C_1}\left(1-\exp\left(\frac{C_1}{\kappa}y\right)\right)\quad \text{for}\quad y<0,
\end{equation*}
for some $ \alpha\in\R $ with $-\frac{L}{K}C_2<\alpha<0 $. Hence, $ \partial_y w(y)=\alpha e^{C_1/\kappa}<0 $

For $ y>0 $ we consider the function \begin{equation*}\label{w+}
	 w(y)=T_Me^{-C_2y}\left(1-\frac{T_M^3}{12 C_2^2}+\frac{T_M^3e^{-3C_2y}}{12 C_2^2}\right) .
\end{equation*} We see that $ w\leq T_Me^{-C_2y} $ as well as $ w\geq 0 $, since $ C_2>\frac{T_M^{3/2}}{2\sqrt{3}} $. Moreover, the function $ w $ satisfies
\begin{multline*}
w''(y)+C_2w'(y)\\= C^2_2T_M e^{-C_2y}\left(1-\frac{T_M^3}{12 C_2^2}\right)+\frac{4}{3}T_M^4e^{-4C_2y}-C^2_2T_M e^{-C_2y}\left(1-\frac{T_M^3}{12 C_2^2}\right)-\frac{1}{3}T_M^4e^{-4C_2y}\\=4T_M^4e^{-4C_2y}=\left(T_Me^{-C_2y}\right)^4\geq w^4.
\end{multline*}
Notice in addition that $ w $ is monotonically decreasing, since 
\begin{equation*}
	w'(y)=-C_2T_M e^{-C_2y}\left(1-\frac{T_M^3}{12 C_2^2}\right)-\frac{1}{3C_2}T_M^4e^{-4C_2y}< 0.
\end{equation*}
Finally, we see that $ \partial_yw(0^+)>-L C_1 $. Indeed, we need to show that 
\begin{equation*}
	w'(0)=-C_2T_M+\frac{T_M^4}{12 C_2}-\frac{1}{3C_2}T_M^4=-\frac{4C_2^2 T_M+T_M^4}{4C_2}>-LC_1,
\end{equation*}
which is equivalent to show that 
\begin{equation*}
	4C_2^2 T_M-4LC_1C_2+T_M^4<0.
\end{equation*}
Since the two roots are given by $ \Gamma^\pm(C_1) =\frac{LC_1\pm\sqrt{L^2C_1^2-T_M^5}}{2}$ and since by assumption $ C_1>\left(\frac{T_M^5+1}{L^2}\right)^{1/2} $ we see that $ \Gamma^\pm(C_1) $ are well defined. Hence, we conclude $ w'(0)>-LC_1 $ using that by assumption
\[\Gamma^-(C_1)<C_2<\Gamma^+(C_1).
\]

Let us now consider $ T_i=T_M+u_i $ the solutions of \eqref{syst.5} considered in Theorem \ref{thm.local.well.pos.} and \ref{thm.loc.existence.reg.} and in Proposition \ref{prop.sol.} on the maximal time interval $ [0,t^*] $. Let us assume that $ t^*<\infty $ and that $ (T_1,T_2,\dot{s}) $ cannot be extended for $ t>t^* $, otherwise it is already the global in time solution. We will show that $ \Arrowvert T_i\Arrowvert_\infty\leq C(w)<\infty $, $ \Arrowvert\partial_yT_i\Arrowvert_\infty<C(w,t^*)<\infty $ and $ \Arrowvert \dot{s}\Arrowvert_\infty< C(w)<\infty $, where the $ \sup $-norms are taken on $ [0,t^*] $. This will imply that the solutions can be extended for $ t>t^* $ as we did in Theorem \ref{thm.local.well.pos.} and \ref{thm.loc.existence.reg.}, and hence $ t^*=\infty $.

Lemma \ref{lemma.global} implies that for the initial value $ T_0 $ as in the assumption of the Theorem there are $ C_1,C_2>0 $ satisfying the prescribed conditions such that $ T_0(y)< w(y) $ for $ y<0 $ and that $ T_0(y)> w(y) $ for $ y>0 $ and such that $ |w'(0^\pm)|>\sup\limits_{\R_\pm}|\partial_y T_0(y|) $. Thus, by the uniform continuity of $ T_i-w\in \mathcal{C}^{1/2,1/2}_{t,y}([0,t_*]\times \R_\pm) $, by the positivity $ \left.\partial_y(T_i(0,y-w))\right|_{y=0}>0 $ as well as by the fact that for any $ R>0 $ there exists $ a>0 $ such that $ |T_0(y)-w(y)|>a $ for all $ |y|>R $, there exists a positive time $ t_0\leq t^* $ such that $ T_1(t,y)<w(y) $ for $ y<0 $ and $ T_2(t,y)>w(y) $ for $ y>0 $ and $ 0\leq t<t_0 $. Let us define 
\[t_0=\inf\{t\in[0,t^*]: \exists y\ne 0 \text{ such that }T_1(t,y)=w(y) \text{ if }y<0 \text{ or }T_2(t,y)=w(y)\text{ if }y>0 \}.\] Let us assume that $ t_0< t^* $. Then, since $ T_1(t,0)=T_M=w(0)=T_2(t,0) $ we have $ 0\geq\partial_y T_1(t,0^-)\geq\partial_y w(0^-) $ as well as $ 0\geq\partial_y T_2(t,0^+)\geq\partial_y w(0^+) $ for $ t\in[0,t_0] $. 
Hence, for $ t\in[0,t_0] $ we obtain that $ -C_1<\dot{s}(t)<C_2 $. Let us also denote by $ \mathcal{L}_1(v)=\partial_t v-\dot{s}(t)\partial_y v-\kappa \partial^2_y v $ for $ y<0 $ and $ \mathcal{L}_2(v)=\partial_t v-\dot{s}(t)\partial_y v- \partial^2_y v+v^4 $ for $ y>0 $. We note that,
\[\mathcal{L}_1(w)=-(\dot{s}(t)+C_1)\partial_y w(y)>0.\]
Hence, $ \mathcal{L}_1(T_1-w)(t,y)<0 $ for all $ (t,y)\in(0,t_0]\times \R_- $ and $ T_1(t,y)-w(y)\leq 0 $ for $ (t,y)\in \{0\}\times \R_-\cup(0,t_0)\times\{a,b\} $, where $ (a,b)\subset \R_- $. An application of the maximum principle to the bounded function $ T_1-w $ on domains $ [0,t_0]\times [a,b] $ for any $ [a,b]\in \R_- $ shows that $ T_1(t,y)< w(t,y) $ for all $ (t,y)\in (0,t_*]\times (a,b) $ for any $ (a,b)\in\R_- $.

Moreover, for $ y>0 $ we see
\[\mathcal{L}_2(w)=-(w''(y)+C_2w'(y)-w^4(y))+(C_2-\dot{s}(t))\partial_y w(y)\leq (C_2-\dot{s}(t))\partial_y w(y)<0 .\]
This implies \[ \mathcal{L}_2(T_2)(t,y)-\mathcal{L}_2(w)(t,y)>\int_0^\infty \frac{\alpha}{2}E_1(\alpha(y-\eta))T^4_2(t,\eta)d\eta>0. \]
Hence, if for $ R>0 $ there exists some $ y_0\in(0,R) $ such that $ T_2(t_0,y_0)-w(y_0)=0 $ we obtain the contradiction
\[0<\mathcal{L}_2(T_2)(t_0,y_0)-\mathcal{L}_2(w)(t_0,y_0)\leq 0.\]
This is because $ \partial_tT_2(t_0,y_0)\leq 0 $  as well as $ T_2(t_0,y_0)-w(y_0)=0 $ would be a minimum in space. 
Thus, $ T_2(t,y)>w(y) $ for all $ (t,y)\in(0,t_0]\times(a,b) $ for every $ (a,b)\in\R_+ $.
Hence, by the definition of $ t_0 $ follows that $ t_0=t^* $ and \[ \Arrowvert T_i\Arrowvert_\infty\leq \max\left\{\Arrowvert w\Arrowvert_\infty,T_M\right\}<\infty.\] Moreover, since $ T_i(t,0)=T_M=w(0) $ we have that $ 0\geq\partial_yT_i(t,0^\pm)\geq \partial_y w(0^\pm) $ and hence by construction \[\Arrowvert \dot{s}\Arrowvert_\infty\leq \max\{C_1,C_2\}<\infty.\] We will now show that also the norms of $ \partial_y T_i $ are bounded. We will use the maximum principle applied to the equation solved by $ \partial_y T_i $. Before considering those equations we shall argue that $ \partial_y T_i $ are indeed twice differentiable. This follows using classic parabolic theory. Let $ \eps>0 $ and $ \gamma>0 $ be arbitrary. We have already shown in Theorem \ref{thm.loc.existence.reg.} that $ T_i\in \mathcal{C}^{1+\delta/2,2+\delta}_{t,y}([\frac{\eps}{2},t^*]\times \R_\pm)\cap \mathcal{C}_{t,y}^{\alpha/2,1+\beta}([0,t^*],\R_\pm) $, for $ \delta\in\left(0,\frac{1}{2}\right] $ and $ \alpha,\beta\in(0,1) $. One can prove that $ T_i\in \mathcal{C}^{1+\delta/2,3+\delta}_{t,y}([\eps,t^*]\times [\pm\gamma,\pm\infty))$ since $ \dot s \partial^2_y T_1 \in \mathcal{C}_{t,y}^{\delta/2,\delta}\left([\frac{\eps}{2},t^*]\times \R_-\right) $ and
\[ \dot s \partial^2_y T_2-4T_2^3\partial_yT_2+\frac{\alpha}{2}T_M^4E_1(\alpha \cdot)+4\int_0^\infty \frac{\alpha}{2}E_1(\alpha(y-\eta))T_2^3(\eta)\partial_\eta T_2(\cdot,\eta)d\eta \in C^{\delta/2,\delta}\left(\left[\frac{\eps}{2},t^*\right]\times \left[\frac{\gamma}{2},\infty\right)\right).\] 
Since the computations are similar to the one in  Proposition \ref{reg.heat}, Lemma \ref{bnd.value.estimates} and Lemma \ref{interior.estimate} we omit the details. Hence, differentiating the equations satisfied by $ T_i $, we obtain that $ \partial_y\partial_t T_i $ exists and it is continuous. Furthermore, differentiating the representation formulas and using classic parabolic theory again, we conclude that the derivatives $ \partial_t\partial_y T_i $ exist and that they are continuous for every $ t,y\in [\eps,t^*]\times[\pm\gamma,\pm\infty) $. Thus, $ \partial_y\partial_t T_i=\partial_y\partial_t T_i $ at the interior of $ (0,t^*]\times\R_\pm $. Differentiating the operators $ \mathcal{L}_1 $ and $ \mathcal{L}_2 $ we see that $ \partial_yT_i $ solve 
\[\mathcal{L}_1(\partial_ yT_1)(t,y)=0 \text{ for } t>0,y<0, {\mathcal{L}}^1_2(\partial_yT_2)(t,y)=\frac{\alpha}{2}T_M^4E_1(\alpha y)>0\text{ and }{\mathcal{L}}^2_2(\partial_yT_2)(t,y)=0 \text{ for } t>0,y>0, \]
where we defined 
\[{\mathcal{L}}^1_2(v)=\partial_t v-\dot{s}(t)\partial_y v- \partial^2_y v+4 T_2^3 v-4\int_0^\infty\frac{\alpha}{2}E_1(\alpha(\cdot-\eta))T_2^3(\eta)v(\cdot,\eta)d\eta \]
as well as \[{\mathcal{L}}^2_2(v)=\partial_t v-\dot{s}(t)\partial_y v- \partial^2_y v+4 T_2^3 v-4\int_0^\infty\frac{\alpha}{2}E_1(\alpha(\cdot-\eta))T_2^3(\eta)v(\cdot,\eta)d\eta-\frac{\alpha}{2}T_M^4E_1(\alpha \cdot). \]
Let us consider for $ t\geq 0 $ the functions $ \psi_\pm(t)=\mp \partial_y w(0^-)(1+t) $. It is easy to see that $ \mathcal{L}_1(\psi_\pm)= \mp \partial_y w(0^-)$ and therefore $ \psi_+ $ is a supersolutions while $ \psi_- $ is a subsolution for $ \mathcal{L}_1 $. Moreover, $ \psi_-<\partial_y T_1(0,y)<\psi_+ $ as well as $ \psi_-<\partial_y T_1(t,0^-)<\psi_+ $ for $ t\in[0,t^*] $ and $ y<0 $. We define
\[t_0=\inf\left\{t_0\in[0,t^*]: \exists y_0<0 \text{ s.t. } \partial_yT_1(t_0,y_0)=\psi_+(t_0)\text{ or }\partial_yT_1(t_0,y_0)=\psi_-(t_0)\right\}.\]
Let us assume that $ t_0<t^* $. Then, by the uniform continuity of $ \partial_yT_1-\psi_\pm\in\mathcal{C}^{1/2,1/2}_{t,y}([0,t^*]\times\R_-) $ and since $ \sup\limits_{y\in\R_-}\partial_yT_1(0,y)>\partial_yw(0^-) $, we know that $ t_0>0 $ as well as $ y_0<0$. Let us assume that $ y_0<0 $ is the smallest such that $ \partial_yT_1(t_0,y_0)=\psi_+(t_0) $. Then, $ \partial_t (\partial_y T_1-\psi_+)(t_0,y_0)\geq 0 $ as well as $ \partial_yT_1(t_0,\cdot) $ has a maximum in $ y_0 $. Thus,
\[0>\mathcal{L}_1(\partial_yT_1-\psi_+)(t_0,y_0)\geq 0.\]
A similar contradiction is obtained applying the maximum principle to $ \partial_y T_1-\psi_- $ assuming that $ \partial_y T_1(t_0,y_0)-\psi_-(t_0)=0 $. Hence, we conclude that $ t_0=t^* $ so that
\[\Arrowvert \partial_yT_1\Arrowvert_\infty\leq |\partial_y w(0^-)|(1+t^*)<\infty.\]
We now consider $ \partial_yT_2 $. Let us define $ \varphi_-= \partial_yw(0^+)e^{4T_M^3t}<0 $. Then, since  $ 0<\int_0^\infty\frac{\alpha}{2}E_1(\alpha(y-\eta))T_2^3(\eta)d\eta\leq T_M^3 $ we see that
\[\mathcal{L}^1_2(\varphi_-)(t,y)\leq 4T_2^3(t,y)\varphi_-(t)<0. \] 
Moreover, $ \varphi_-(0)<\sup\limits_{y\in\R_+}\partial_y T_2(0,y) $ for $ y>0 $ as well as $ \varphi_-(t)<\partial_y T_2(t,0^-)\leq0 $ for $ t\in[0,t^*] $. We remark that $ \partial_yT_2-\varphi_-\in\mathcal{C}^{1/2,1/2}_{t,y}([0,t^*]\times\R_+) $. Defining again via uniform continuity \[t_0=\inf\left\{t_0\in[0,t^*]: \exists y_0>0 \text{ s.t. } \partial_yT_2(t_0,y_0)=\varphi_-(t_0)\right\}>0,\]
assuming $ t_0<t^* $ and applying the maximum principle to $ {\mathcal{L}}^1_2(\partial_y T_2-\varphi_-)(t,y)>0 $ at $ (t_0,y_0) $ we obtain a contradiction. Indeed, $ (\partial_y T_2-\varphi_-)(t_0,y)\geq 0  $ so that $ (\partial_y T_2-\varphi_-)(t_0,y_0) $ is a minimum. Hence,
\[0<{\mathcal{L}}^1_2(\partial_y T_2-\varphi_-)(t_0,y_0)\leq -4 \int_0^\infty\frac{\alpha}{2}E_1(\alpha(y-\eta))T_2^3(\eta)(\partial_y T_2-\varphi_-)(t_0,\eta)d\eta<0. \]
This contradiction implies $ \partial_y T_2(t,y)\geq \partial_yw(0^+)e^{4T_M^3t^*} $ for all $ t\in[0,t_*] $. 
Let us now define \[g(y)=-T_M^4\int_0^yd\xi\int_0^\xi dz e^{-(\xi-z)}\frac{\alpha}{2}E_1(\alpha(z))\eta(z),\]
where $ \eta\in C^\infty([0,\infty)) $ with $ 0\leq \eta\leq 1 $, $ \eta(z)\equiv 1 $ for $ y\in\left[0,\frac{1}{2}\right] $ and $ \eta(z)\equiv 0 $ for $ y\geq 1 $. A simple computation shows $ -T_M^4\leq g(y)\leq 0$ as well as $ -\frac{T_M^4}{2}\leq g'(y)\leq 0 $. We remark that $ g\in C^{0,1/2}(\R_+) $. Moreover, for $ y>0 $ the function $ g $ solves
\[-g''(y)-g'(y)=T_M^4\frac{\alpha}{2}E_1(\alpha(y))\eta(y).\] 
We also consider the function $ h\in C^{0,1/2}([0,t^*]) $ given by
\[\begin{split}
	h(t)=&\left[|\partial_yw(0^+)|+T_M^4+\frac{T_M}{4}\left(1+\frac{C_1+1}{2}+4T_M^3\right)\right]e^{4T_M^3t}-\frac{T_M}{4}\left(1+\frac{C_1+1}{2}+4T_M^3\right)\\\geq& \left[|\partial_yw(0^+)|+T_M^4\right]e^{4T_M^3t}\geq\left[|\partial_yw(0^+)|\right]e^{4T_M^3t}+|g(y)|>0.
\end{split}\] 
Using the estimates $ \frac{\alpha}{2}E(\alpha(y))\leq 1 $ for all $ y\geq \frac{1}{2} $ and $ -\dot{s}(t)\leq C_1 $ we compute for $ \varphi_+(t,y)=h(t)+g(y)\geq |\partial_yw(0^+)|e^{4T_M^3t}>\sup\limits_{y\in\R_+}\left|\partial_y T_2(0,y)\right|>0 $ the following
\[\begin{split}
	\mathcal{L}_2^2(\varphi_+)(t,y)=&\partial_t h(t)+4T_2^3(t,y)(\varphi_+)(t,y)-\int_0^\infty\frac{\alpha}{2}E_1(\alpha(y-\eta))T_2^3(\eta)\varphi_+(t,\eta)d\eta\\&+(-\dot{s}(t)+1)\partial_yg(y)+\frac{\alpha}{2}T_M^4E_1(\alpha y)(\eta(y)-1)\\
	>&\partial_t h(t)-4T_M^3 T_M^4-(C_1+1)\frac{T_M^4}{2}-T_M^4>0.\\
\end{split}\] 
For this estimate we used that $ g\leq 0 $ and $ h>0 $. We also notice that by construction $ \partial_yT_2(0,y)<\varphi_+(t,y) $ as well as $ \partial_y T_2(t,0)\leq0<\varphi_+(t,y) $ for $ t\in[0,t^*] $ and $ y>0 $. Using the uniform continuity of $ \partial_yT_2-\varphi_+\in\mathcal{C}^{1/2,1/2}_{t,y}([0,t^*]\times\R_+) $ we define \[t_0=\inf\left\{t_0\in[0,t^*]: \exists y_0<0 \text{ s.t. } \partial_yT_2(t_0,y_0)=\varphi_+(t_0,y_0)\right\}>0.\]
Assuming $ t_0<t^* $, arguing by continuity and applying once more the maximum principle to $ {\mathcal{L}}^2_2(\partial_y T_2-\varphi_+)(t,y)<0 $ at $ (t_0,y_0) $ we obtain a contradiction. 
Indeed, we use that $ (\partial_y T_2-\varphi_+)(t_0,y)\leq 0  $ and therefore $ (\partial_y T_2-\varphi_+)(t_0,y_0) $ is a maximum. Thus,
\[0>{\mathcal{L}}^2_2(\partial_y T_2-\varphi_+)(t_0,y_0)\geq0.\]
We finally conclude that 
\[\Arrowvert \partial_yT_2\Arrowvert_\infty\leq \left[|\partial_yw(0^+)|+T_M^4+\frac{T_M}{4}\left(1+\frac{C_1+1}{2}+4T_M^3\right)\right]e^{4T_M^3t^*}.\]
Therefore, $ (T_1,T_2,\dot{s})$ can be extended for $ t>t^* $.
\end{proof}

Finally, we prove Lemma \ref{lemma.global}. 
\begin{proof}[Proof of Lemma \ref{lemma.global}]
By our assumptions on $ T_0 $ we can fix some $ \theta\in(0,1) $ such that $ \sup\limits_{\R_-} T_0\leq T_M+\frac{L^2 \kappa}{T_M K}\frac{(1-\theta)}{4} $. Let us also define \[ C^0_2>\max\left\{\frac{T_M^{3/2}}{2\sqrt{3}} ,\Gamma^-\left(\left(\frac{T_M^5+1}{L^2}\right)^{1/2}\right),\frac{T_M^{3/2}}{2},\frac{\left(T_M^5+1\right)^{1/2}}{2T_M},\sup\limits_{\R_+}\frac{|\partial_yT_0|}{T_M},\sup\limits_{\R_-}\frac{|\partial_yT_0|K}{L(1-\theta)} \right\}\]
and let us denote $ f_-^{C_2^0}(y)=T_M+\frac{(1-\theta)L^2\kappa}{2T_MK}\left(1-e^{\frac{2T_M C_2^0}{L\kappa}y}\right) $ for $ y<0 $ and $ f_+^{C_2^0}(y)=T_Me^{-C_2^0y} $ for $ y>0 $. Then, since $ 0\geq\partial_yT_0(0^-)>\partial_yf_-^{C_2^0}(0^-)  $ and $0\geq\partial_yT_0(0^+)>\partial_yf_+^{C_2^0}(0^+)  $, by monotonicity there exist constants $ \delta_1(C_2^0),\delta_2(C_2^0)>0 $ such that
\[T_0(y)<f_-^{C_2^0}(y)\text{ for }y\in(-\delta_1(C_2^0),0) \text{ and }f_+^{C_2^0}(y)<T_0(y)\text{ for }y\in(0,\delta_2(C_2^0)).\]
We remark that since $ C_2\mapsto f_-^{C_2}(y) $ is increasing in $ C_2>0 $ for $ y<0 $ and since $ C_2\mapsto f_+^{C_2}(y) $ is decreasing in $ C_2>0 $ for $ y>0 $  the estimates $ T_0(y)<f_-^{C_2}(y) $ and $ f_+^{C_2}(y)<T_0(y) $ are valid in the intervals $ (-\delta_1(C_2^0),0) $ and $ (0,\delta_2(C_2^0)) $, respectively. Thus, defining
\[C_2=\max\left\{C_2^0,\frac{L\kappa}{2T_M\delta_1(C_2^0)}\ln\left(\frac{1}{2}\right), \frac{1}{\delta_2(C_2^0)}\ln\left(\frac{2T_M}{\inf\limits_{\R_+}T_0}\right)\right\}\]
 we see that $ f_-^{C_2}(y)>T_0(y) $ for $ y<0 $ as well as $ T_0(y)>f_+^{C_2}(y) $ for $ y>0 $. Moreover, we notice that $ C_2<\infty $ by assumption on $ T_0 $. We now choose $ \alpha=-\frac{(1-\theta)LC_2}{K} $ and $ C_1=\frac{2T_M}{L}C_2 $. Notice that by definition $ C_1>\left(\frac{T_M^5+1}{L^2}\right)^{1/2} $. Then, for the chosen constants $  C_1,C_2,\alpha $ the function $ w $ defined by \eqref{W} satisfies \[w(y)=f_-^{C_2}(y)\text{ for }y<0\text{ and }w(y)<f_+^{C_2}(y)\text{ for }y>0.\]
 Thus, we have found constants $ C_1,C_2>0 $ such that
 \[w(y)<T_0(y )\text{ for }y<0\text{ and }T_0(y)>0\text{ for }y>0.\] Notice that by construction for any $ R>0 $ there exists $ a>0 $ such that $ |T_0(y)-w(y)|>a $ for all $ |y|>R $.
 Moreover, by definition
 \[|\partial_y w(0^-)|=|\alpha|=\frac{(1-\theta)LC_2}{K}>\sup\limits_{\R_-}|\partial_yT_0| \text{ and }|\partial_y w(0^+)|>C_2T_M>\sup\limits_{\R_+}|\partial_yT_0(y)|.\]
 Finally, since $ C_2>\frac{T_M^{3/2}}{2} $ we conclude
\[\partial_yw(0^+)=-C_2T_M\left(1+\frac{T_M^3}{4C_2^2}\right)>-2C_2T_M=-LC_1,\] which implies $ \Gamma^-(C_1)<C_2<\Gamma^+(C_1) $ as well as the fact
\[w(0^-)=\alpha=-\frac{(1-\theta)LC_2}{K}>-\frac{L}{K}C_2.\] 

We also remark that the considered class of initial data is optimal for the argument in Theorem \ref{thm.global} involving $ w $ as a barrier function defined in \eqref{W}. Since we can take $ C_2 $ arbitrary large as $ C_2=\frac{1}{\eps}\to \infty $ we obtain \[ w(y)\leq T_Me^{-\frac{y}{\eps}} \text{ for }y>0\text{ and } \partial_yw(0^+)=-\frac{T_M}{\eps}\left(1+\frac{T_M^3\eps^2}{4}\right)\to-\infty \text{ as }\eps\to 0.\]
Moreover, $ C_1> \frac{T_M}{L\eps}\left(1+\frac{T_M^3\eps^2}{4}\right) $ so that $ \partial_y w(0^-)=-|\alpha|>-\frac{K}{L\eps}\to -\infty $ as $ \eps\to0$ and
\[  w(y)< T_M+\frac{\kappa L}{KC_1}\frac{1}{\eps}\left(1-\exp\left(\frac{C_1}{\kappa}y\right)\right)<T_M+\frac{\kappa L^2}{KT_M}\frac{1}{\left(1+\frac{T_M^3\eps^2}{4}\right)}\left(1-\exp\left(\frac{C_1}{\kappa}y\right)\right)\to T_M+\frac{\kappa L^2}{KT_M}\]
as $ \eps \to 0 $ and $ y<0 $. Moreover, $ \dot{s}(t)\in (-\infty,\infty) $.
	
\end{proof}

\begin{remark}

Observe also that the class of initial values $ T_0 $ considered in Theorem \ref{thm.global} is optimal for the argument presented in the proof of the Theorem. Indeed, instead of considering for $ y>0 $ the subsolution $ w''+C_2w'\geq w^4 $, we could have considered the solution to
\begin{equation}\label{ode1}
	\begin{cases}
		w''(y)+C_2w'(y)=w^4(y)&y>0\\
		w(0)=T_M\\
		w(y)\to 0 & \text{ as }y\to\infty\\
	w\geq 0	\end{cases}
\end{equation}
That such a solution exists can be proven considering the variational problem given by the functional
\[I(w)=\int_0^\infty e^{C_2y}\left(\frac{|\partial_y w(y|^2)}{2}+\frac{w^5(y)}{5}\right) dy.\]
By the direct method of calculus of variations one can prove that there exists a unique minimizer of $ I $ on the closed convex set
\[\mathcal{A}= \left\{g\geq 0: g\in W^{1,2}\left(\R_+;e^{C_2y}dy\right)\cap L^5\left(\R_+;e^{C_2y}dy\right),\; g(0)=T_M\right\}.\]
Since then $ e^{\frac{C_2}{2}y}g\in L^\infty\cap C^{0,1/2}(\R_+) $ it also follows that if $ g\in\mathcal{A} $ then $ \lim\limits_{y\to\infty}g=0 $. The unique minimizer $ w\in\mathcal{A} $ is also bounded by $ T_M $, since $ I[\min\{w,T_M\}]\leq I[w]\leq I[\min\{w,T_M\}] $. Moreover, $ w $ solves weakly the following variational inequality
\[-\partial_y\left(e^{C_2y}\partial_y w(y)\right)+w^4(y)e^{C_2y}\geq 0,\]
and is a weak solution of $ -\partial_y\left(e^{C_2y}\partial_y w(y)\right)+w^4(y)e^{C_2y}= 0 $ in the region $ w>0 $. With the weak maximum principle it can be also shown that $ \{y>0:w(y)>0\}=\R_+ $. This implies that the unique minimizer $ w $ is a weak solution of 
\[ -\partial_y\left(e^{C_2y}\partial_y w(y)\right)+w^4(y)e^{C_2y}= 0 \text{ in }\R_+.\]
Using elliptic regularity we obtain easily that since $ w\in L^\infty\cap C^{0,1/2}(\R_+)  $ also $ w^4\in L^\infty\cap C^{0,1/2}(\R_+) $ and hence locally $ w\in C_{\text{loc}}^{2+1/2}(\R_+) $ so that iterating this argument we have $ w\in C^\infty(\R_+)\cap L^\infty(\R_+) $. Thus, $ w $ is a strong solution solving the boundary problem \eqref{ode1}.

Finally, the solution $ w $ of \eqref{ode1} is unique. This is a consequence of the strong maximum principle.
\\

Let us assume now again that $ C_2=\frac{1}{\eps} $ is arbitrarily large. Then using the rescaling $ y=\eps\eta $ and $ w(y)=w(\eps\eta)=\tilde{w}(\eta) $ we see that the leading order of $ \tilde{w} $ solves as $ \eps\to 0 $
\begin{equation*}\label{ode2}
	\begin{cases}
		\tilde{w}''(\eta)+	\tilde{w}'(\eta)=0&\eta>0\\
			\tilde{w}(0)=T_M\\
			\tilde{w}(\eta)\to 0 & \text{ as }\eta\to\infty\\
			\tilde{w}\geq 0	\end{cases}
\end{equation*}
Hence, $ \tilde{w}(\eta)=T_M\left(e^{-\eta}\right) $. Thus, $ w(y)=T_M\left(e^{-y/\eps}\right)  $ at the leading order, so that we need to take \[C_1>\frac{T_M}{\eps L},\]
which implies for $ y<0 $ as above that \[ w(y)< T_M+\frac{\kappa L}{KC_1}\frac{1}{\eps}<T_M+\frac{\kappa L^2}{KT_M} \text{ as }\eps\to 0 .\]
\end{remark}

\bibliographystyle{siam}
\bibliography{literature_stefan}
\end{document}